\documentclass[12pt]{article}

\usepackage{url}
\usepackage{mathtools}
\usepackage{amssymb}
\usepackage{amsthm}
\usepackage{empheq}
\usepackage{latexsym}
\usepackage{enumitem}
\usepackage{eurosym}
\usepackage{dsfont}
\usepackage{appendix}
\usepackage{color} 
\usepackage[unicode]{hyperref}
\usepackage{frcursive}
\usepackage[utf8]{inputenc}
\usepackage[T1]{fontenc}
\usepackage{geometry}
\usepackage{multirow}
\usepackage{todonotes}
\usepackage{lmodern}
\usepackage{anyfontsize}
\usepackage{stmaryrd}
\usepackage{natbib}
\usepackage{cleveref}

\setcitestyle{numbers,open={[},close={]}}

\definecolor{red}{rgb}{0.7,0.15,0.15}
\definecolor{green}{rgb}{0,0.5,0}
\definecolor{blue}{rgb}{0,0,0.7}
\hypersetup{colorlinks, linkcolor={red},citecolor={green}, urlcolor={blue}}
			
\makeatletter \@addtoreset{equation}{section}

\newtheorem{theorem}{Theorem}[section]
\newtheorem{assumption}[theorem]{Assumption}

\newtheorem{example}[theorem]{Example}

\newtheorem{lemma}[theorem]{Lemma}
\newtheorem{proposition}[theorem]{Proposition}

\newtheorem{definition}[theorem]{Definition}
\newtheorem{remark}[theorem]{Remark}

\DeclareUnicodeCharacter{014D}{\=o}
\def \E{\mathbb{E}}
\def \F{\mathbb{F}}
\def \G{\mathbb{G}}

\def \J{\mathbb{J}}

\def \L{\mathbb{L}}

\def \N{\mathbb{N}}

\def \P{\mathbb{P}}
\def \Q{\mathbb{Q}}
\def \R{\mathbb{R}}

\def \V{\mathbb{V}}

\def \X{\mathbb{X}}

\def\Ac{{\cal A}}
\def\Bc{{\cal B}}
\def\Cc{{\cal C}}

\def\Ec{{\cal E}}
\def\Fc{{\cal F}}

\def\Lc{{\cal L}}

\def\Pc{{\cal P}}

\def\Sc{{\cal S}}

\def\Uc{{\cal U}}

\def\x{\times}
\def\eps{\varepsilon}

\def\Om{\Omega}

\def\om{\omega}

\def\0{\mathbf{0}}
\def \Ec{\mathcal{E}}

\def \Xh{\widehat{X}}

\def\normeL2#1{\left\|{#1}\right\|_{L^2}}

\def\1{\mathbf{1}}

\def \proof{{\noindent \bf Proof.\quad}}
\def \ep{\hbox{ }\hfill{ ${\cal t}$~\hspace{-5.1mm}~${\cal u}$}}

\def\b*{\begin{eqnarray*}}
\def\e*{\end{eqnarray*}}

\def\be{\begin{eqnarray}}
\def\ee{\end{eqnarray}}

\def\drm{\mathrm{d}}

\def\supp{\mbox{supp}}

\title{\bf Mean Field Game of Mutual Holding with common noise
      \thanks{This work benefitted from the financial support of the Chairs {\it Financial Risk} and {\it Finance and Sustainable Development}.}}
\author{Leila Bassou\thanks{Ecole Polytechnique, 
            CMAP,
            leila.bassou@polytechnique.edu}
            \and 
            Mao Fabrice Djete\thanks{Ecole Polytechnique, 
            CMAP,
            mao-fabrice.djete@polytechnique.edu. }
            \and 
            Nizar Touzi\thanks{New York University, 
            Tandon School of Engineering, nizar.touzi@nyu.edu. }}

\date{\today}

\begin{document}

\maketitle
 
\begin{abstract}
We consider the mean field game of cross--holding introduced in \citeauthor*{DjeteTouzi} \cite{DjeteTouzi} in the context where the equity value dynamics are affected by a common noise. In contrast with \cite{DjeteTouzi}, the problem exhibits the standard paradigm of mean--variance trade off. Our crucial observation is to search for equilibrium solutions of our mean field game among those models which satisfy an appropriate notion of no--arbitrage. Under this condition, it follows that the representative agent optimization step is reduced to a standard portfolio optimization problem with random endowment.  
\end{abstract}

\vspace{3mm}

\vspace{3mm}
\section{Introduction}

Modeling financial interactions between economic entities is a crucial step in order to gain some understanding of the major question of financial stability of the economic sphere. The equity value of an economic entity depends on one hand on the revenues generated by its specific business structure. On the other hand, they also need to access to appropriate funding for their development. In addition to this first source of interaction through debt, economic entities may also interact with each other through equity cross--holding for risk diversification purpose by appropriate sharing of their profits--and--losses.

This paper follows the model introduced in \cite{DjeteTouzi} and \citeauthor*{DjeteGuoTouzi} \cite{DjeteGuoTouzi} where the interaction through debt is ignored. Similar to these works, we are interested in characterizing the Nash equilibria of the problem of optimal cross--holding. We consider a large population of symmetric agents who seek to optimize their equity value by optimally choosing the level of holding their peers, while undergoing the level of holding its own assets by them. This induces an optimal response function for each agent, and we search for a Nash equilibrium of this cross--holding problem by analyzing the fixed points of the optimal responses of all agents. 

Similar to \cite{DjeteTouzi}, we consider the mean field version of this problem which is more accessible due to the fact that the limiting model erases the impact of the representative agent on the population distribution because one single agent becomes negligible in the infinite population limit. 

Our main objective in this paper is to allow for correlated revenue structures of the population of firms. We model this by introducing a common factor of risk which affects their idiosyncratic risk process. Our main results are the following:
\begin{itemize}
\item The existence of a solution for the mean field game of cross holding is intimately related to an appropriate notion of no--arbitrage which must be satisfied at equilibrium. We obtain an explicit characterization of this condition in terms of the model ingredients, which is in the same spirit as the Heath--Jarrow--Morton \cite{heath1992bond} restriction on the drift in the interest rates modeling literature.
\item Under the no--arbitrage condition, we provide a characterization of the set of equilibrium solutions for the mean field game of cross--holding.
\item Explicit examples are obtained in the context of the Black--Scholes model with common--noise depending coefficients. 
\end{itemize}
To be more specific, let us start by the simple one--period model where the variation of the equity value $\Delta X^i:=X^1_1-X^i_0$ in a zero--interest rate context is given by
$$
\Delta X^i
=
b(X^i_0)
+\sigma(X^i_0)\varepsilon_i
+\sigma^0(X^i_0)\varepsilon^0
+\frac1N\sum_{1\le j\ne i\le N}\beta(X^i_0,X^j_0)\Delta X^j
                                    -\pi(X^j_0,X^i_0)\Delta X^i,
$$
for $i=1,\ldots,N$,
where $b,\sigma,\sigma^0$ are given deterministic functions, and $\eps^0,\eps_1,\ldots,\eps_N$ are independent centered random variables with unit variance, and the $\eps_i$'s are independent copies of some r.v. $\eps$. The corresponding mean field version models the variation of the representative equity value as:  
\begin{equation}\label{1periodMF}
\Delta X
=
b(X_0)
+\sigma(X_0)\varepsilon
+\sigma^0(X_0)\varepsilon^0
+\widehat\E^\mu\big[\beta(X_0,\widehat X_0)\Delta\widehat X\big]
-\widehat\E^\mu\big[\pi(X_0,\widehat X_0)\big]\Delta X,
\end{equation}
where we denoted $\widehat\E^\mu[f(X,\widehat X)]=\displaystyle\int_{\R} f(X,\widehat{x})\mu(\drm\widehat{x})$, and $\mu$ denotes the equilibrium conditional distribution of the population equity value $\widehat X=(\widehat X_0,\widehat X_1)$ conditional to the common noise $\varepsilon^0$. Given its performance criterion, the representative agent's optimal response $\widehat\beta(\mu,\pi)$ is the set of the solution of the performance maximization over $\beta$, given the environment $(\mu,\pi)$. A solution of the mean field game is a pair $(\mu,\pi)$ such that $\pi\in\widehat\beta(\mu,\pi)$ with $\mu$ as the distribution of the corresponding $X=(X_0,X_1)$.

In the continuous time version of this problem, the dynamics of the representative equity value process is defined by
\be\label{continuoustimeMF}
\drm X_t
&=&
\drm P_t
+\widehat\E^\mu\big[\beta_t(X_t,\widehat X_t)\drm\widehat X_t\big]
-\widehat\E^\mu\big[\pi_t(X_t,\widehat X_t)\big]\drm X_t,
\\
\drm P_t
&=&
b_t(X_t)\drm t
+\sigma_t(X_t)\drm W_t
+\sigma^0_t(X_t)\drm W^0_t,
\nonumber
\ee
where $W^0,W$ are independent scalar Brownian motions, and $\widehat\E^\mu$ denotes now the conditional expectation on the continuous paths space of the copy $\widehat X$ at equilibrium with respect to law $\mu$, conditional on the common noise $W^0$.

In the no--common noise situation $\sigma^0=0$ studied in \cite{DjeteTouzi,DjeteGuoTouzi}, the equilibrium distribution $\mu$ is a deterministic object and the solution of the problem is derived by guessing that the equilibrium dynamics of the environment is of the form $\drm\widehat X_t=B_t(\widehat X_t)\drm t+\Sigma_t(\widehat X_t)\drm\widehat W_t$. The critical observation is that under this guess, the continuous time dynamics of the controlled representative equity process is 
$$
\drm X_t
=
\drm P_t
+\widehat\E^\mu\big[\beta_t(X_t,\widehat X_t)B_t(\widehat X_t)\big]\drm t
-\widehat\E^\mu\big[\pi_t(X_t,\widehat X_t)\big]\drm X_t,
$$
where the ``$\drm\widehat W_t$'' term vanishes under the $\widehat\E^\mu-$expectation. We are then reduced to a problem of linear drift control leading to a standard bang--bang type of solution which consists in maximizing the instantaneous return only, and thus does not follow the usual paradigm in portfolio optimization of balancing mean and variance of returns.

In contrast, by allowing for the common noise $\eps^0$ and $W^0$ in the one period and the continuous time model, respectively, the situation becomes more intriguing due to the randomness of the conditional distribution $\mu$. The equilibrium dynamics of the environment are now expected to be of the form $\drm\widehat X_t=B_t(\widehat X_t)\drm t+\Sigma_t(\widehat X_t)\drm \widehat W_t+\Sigma^0_t(\widehat X_t)\drm W^0_t$, leading to the following continuous time dynamics of the controlled representative equity process:
$$
\drm X_t
=
\drm P_t
+\widehat\E^\mu\big[\beta_t(X_t,\widehat X_t)B_t(\widehat X_t)\big]\drm t
+\widehat\E^\mu\big[\beta_t(X_t,\widehat X_t)\Sigma^0_t(\widehat X_t)\big]\drm W^0_t
-\widehat\E^\mu\big[\pi_t(X_t,\widehat X_t)\big]\drm X_t.
$$
Notice that in the present context, unlike the no--common noise setting, the control $\beta$ impact both the drift and the diffusion coefficient, and we are therefore back to the usual paradigm of mean--variance tradeoff in portfolio optimization.

The first main results of this paper is that any equilibrium solution of the above mean field game of cross-holding must satisfy an appropriate notion of no--arbitrage, namely the no--increasing profit (NIP) condition as appeared in the previous literature, see e.g. \cite{Fontana}. Moreover, we provide a characterization of the NIP condition by means of a proportionality relation between instantaneous expected mean $b$ and the common noise volatility $\sigma^0$. Finally, by restricting the search of equilibria to those models which satisfy NIP condition, we provide a necessary and sufficient condition for the existence of an equilibrium solution of the mean field game which can be derived explicitly in some examples. These results are established both in the one--period and the continuous time model. 

Actually, in our context, the NIP condition reduces the representative agent's optimization step to a standard portfolio optimization problem for a single agent with random endowment. In particular, the strategic cross--holding activity within the population does not induce a mean field interaction in the equilibrium dynamics. This is another significant difference with the results of \cite{DjeteTouzi,DjeteGuoTouzi}. However, we observe that the mean field interaction in the equilibrium dynamics of \cite{DjeteTouzi,DjeteGuoTouzi} is fundamentally induced by the constraint on the strategies $\beta$ to take values in $[0,1]$ due to the bang--bang feature of the problem in the no--common noise setting.

\section{The one--period mean field game of cross--holding} 
\label{sec:oneperiod}

Our starting point in the section is the mean field one--period dynamics \eqref{1periodMF} of the representative equity price variation, which we may write by the tower property as:
\begin{equation}\label{1periodMF1}
\big(1+\widehat\E^{\mu_0}[\pi( \widehat X_0, X_0)]\big)\Delta X
=
b(X_0)
+\sigma(X_0)\varepsilon
+\sigma^0(X_0)\varepsilon^0
+\widehat\E^{\mu_0}\big[\beta(X_0,\widehat X_0)F^\mu(\widehat X_0,\eps^0)\big],
\end{equation}
where $\widehat\E^{\mu_0}\big[\psi(X_0,\widehat X_0,\eps^0)\big]=\int_{\R}\psi(X_0,\widehat x_0,\eps^0)\mu_0(\mathrm{d}\widehat x_0)$,  $\mu$ is the conditional law of the pair $(X_0,X_1)$, conditional on the common noise $\eps^0$, $\widehat\E^{\mu}$ denotes the expectation operation with respect to the random measure $\mu$ on the copy space, and
\be
F^\mu(\widehat X_0,\eps^0)
&:=&
\widehat\E^\mu\left[\Delta\widehat X|\widehat X_0\right].
\ee
Here, $(\mu,\pi)$ define a random environment representing the equity value distribution conditional on the common noise and the cross--holding strategy of the surrounding population, respectively. Notice that $\mu$ impacts the dynamics of $(X_0,X_1)$ only through the conditional mean of value variation $F^\mu\in\L^1(\mu_0\otimes\rho)$. Consequently, the random environment in the one--period model reduces to the pair $(F^\mu,\pi)$. We shall denote by 
$$
X_1^{\mu,\pi,\beta}(X_0):=X_0+\Delta X^{\mu,\pi,\beta}(X_0),
$$ 
where the increment $\Delta X^{\mu,\pi,\beta}(X_0):=\Delta X$ as given by \eqref{1periodMF1}. Throughout this section, we assume for simplicity that 
\begin{equation}
    b,\sigma,\sigma^0\;\mbox{ bounded and }\sigma(X_0),\sigma^0(X_0)\mbox{ non--degenerate, }\P\mbox{--a.s.}
\end{equation}

Our main concern is to find a mean field game solution in the current context where the representative agent seeks to maximize the criterion
\b*
J(X_0,\beta;F^\mu,\pi)
:=
\Uc\left( \Lc \left(X_1^{\mu,\pi,\beta}(X_0) | X_0 \right)\right),
\e*
for some function $\Uc:\Pc_p(\R)\longrightarrow\R$, for some $p>0$. Of course, this requires to restrict the admissible strategies to the collection $\Ac_p(F^\mu,\pi)$ of all cross--holding strategies $\beta$ which induce an $\L^p$--integrable random variable $X_1^{\mu,\pi,\beta}(X_0)$.
\begin{definition}[One period mean field game of cross--holding]\label{def:MFG1}
A random environment $(F^\mu,\pi)$ is an equilibrium solution of the mean field game of cross--holding if $F^\mu\in\L^1(\mu_0\otimes\rho)$, $\pi\in\Ac_p(F^\mu,\pi)$, and $\P$--a.e.
\begin{itemize}
\item[{\rm (i)}] $J(X_0,\pi;F^\mu,\pi)=\max_{\beta\in\Ac_p(F^\mu,\pi)} J(X_0,\beta;F^\mu,\pi)$,
\item [{\rm (ii)}] $\E\big[\Delta X_1^{\mu,\pi,\pi}(X_0)|X_0,\eps^0\big]=F^\mu(X_0,\eps^0)$.
\end{itemize}

\end{definition}

\subsection{No--arbitrage (NA)}\label{NA}

In the context of the current one--period model, we introduce the following notion of no--arbitrage which appears naturally as a necessary condition for the individual optimization of Step (i) of Definition \ref{def:MFG1} to have a solution. As a consequence, it is sufficient to limit our search of equilibrium solutions of the mean field game of cross-holding to those environments which satisfy the following no--arbitrage condition. Denote
   \begin{equation*}
   G_{\mu}^\beta=
       G_{\mu}^\beta(X_0,\varepsilon^0)
       :=
       \widehat \E^{\mu_0}\big[ \beta(X_0,\widehat X_0)F^\mu(\widehat X_0, \varepsilon^0)\big] 
       =
       \int_{\R} \beta(X_0,\widehat x_0)F^\mu(\widehat x_0, \varepsilon^0)\mu_0(\mathrm{d}\widehat x_0).
   \end{equation*}

\begin{definition}[No--arbitrage (NA)]\label{NA_def} 
Let  $(F^\mu,\pi)\in\L^1(\mu_0\!\otimes\!\rho)\times \L^1(\mu_0\!\otimes\!\mu_0)$. We say that  $(F^\mu,\pi)$ satisfies the no--arbitrage condition if for all $\beta \in \Ac_1(F^\mu,\pi)$, we have:  
\begin{equation}G_{\mu}^\beta \geq 0,\, \mu_0\!\otimes\!\rho\mbox{--a.s} \implies  G_{\mu}^\beta=0,\,   \mu_0\!\otimes\!\rho\mbox{--a.s}.\label{condition_NA}
\end{equation}
\end{definition} 
As standard, the last condition can be reformulated as 
 \be \label{condition_K_NA}
K \cap \L^1_+(\mu_0 \otimes \rho)
&=&
\{0\},
 \ee
where $K$ is the collection of all (super--)hedgeable claims
\begin{equation}\label{ensemble_K}
 K
 :=
 \Big\{ \xi\!=\!\xi(X_0,\varepsilon^0) \!\in\! \L^1(\mu_0 \otimes \rho): \xi \leq G_{\mu}^{ \beta},
 ~\mbox{a.s. for some}~\beta \!\in\! \Ac_1(F^{\mu},\pi)
 \Big\}.
 \end{equation}
The following result is proved by adapting the standard methods in the no--arbitrage literature.

\begin{lemma}\label{lemme_fermeture}
Under the no--arbitrage condition \eqref{condition_K_NA}, the set $K$ is closed in $\L^1$.
\end{lemma}

The proof is reported in Appendix \ref{preuves_section2}. The following result provides a characterization of the NA condition by means of a collinearity relationship between the drift $b$ and the volatility of the common noise $\sigma^0$. 
This restriction is similar to the drift--volatility restriction in the HJM term structure  model. The latter restriction is a consequence of the fact that self--financing strategies in the HJM model are built on the financial market containing an infinite number of assets, namely a zero--coupon bond for all maturity $T>0$. Notice that the analogy with the HJM model is natural as our model involves trading in an infinite number of assets defined as the equity value of each agent of the surrounding population.

\begin{theorem}\label{NA_Theorem}
For a random environment $(F^\mu,\pi)\in\L^1(\mu_0\!\otimes\!\rho)\times \L^1(\mu_0\!\otimes\!\mu_0)$, the following statements are equivalent: 
\begin{enumerate}[label=\upshape(\roman*),ref= (\roman*)]
    \item\label{th:first} $(F^\mu,\pi)$ satisfies the {\rm NA} condition;
    \item\label{th:second} There exists an $\varepsilon^0$--measurable r.v. $Z$ such that 
    \begin{equation}
Z>0, \, \,  \rho\mbox{--a.s}, \, \, \E^{\rho}[Z]=1 \mbox{ and } \, \E^{\rho}\big[Z  F^\mu(x_0,\varepsilon^0) \big]=0, \, \, \mu_0(\mathrm{d}x_0)\mbox{--a.s}. 
\label{forme_Z}
\end{equation}
    \end{enumerate}
If in addition $(\mu,\pi)$ is an equilibrium solution of the MFG, then with the r.v. $Z$ of {\rm (ii)},   the last conditions are equivalent to 
\begin{equation}
\label{condition_coli}
b(x_0)+ \E^{\rho}[Z \varepsilon^0]\sigma^0(x_0)=0,\, \,  \mu_0(\mathrm{d}x_0)\mbox{--a.s}.
\end{equation}

\end{theorem}

\proof We shall denote throughout $ m^{\pi}(X_0)^{-1}:=1+\widehat\E^{\mu_0}[\pi(\widehat X_0,X_0)]$.

\smallskip 
\noindent (i)$\Longrightarrow$(ii): By Lemma \ref{lemme_fermeture}, $K$ is a closed convex cone containing $\L^1_-(\mu_0 \otimes \rho)$ satisfying \eqref{condition_K_NA}. Then, it follows from Yan's Theorem \cite{yan} that there exists an $(\varepsilon^0,X_0)$--measurable random variable $\tilde Z \in \L^\infty_+(\mu_0 \otimes \rho)$ satisfying $\E^{\mu_0 \otimes \rho}\big[ \tilde Z(X_0, \varepsilon^0) G_{ \mu}^{ \beta}(X_0,\varepsilon^0)\big] \le 0$ for all $\beta\in\Ac_1(F^\mu,\pi)$, which we may rewrite by Fubini's Theorem as:
$$ 
\E^{\mu_0 \otimes \rho}\big[ \tilde Z(X_0, \varepsilon^0) G_{ \mu}^{ \beta}(X_0,\varepsilon^0)\big]=\E^{\mu_0^{\otimes 2}}\Big[
  \beta(X_0,\widehat X_0)\E^\rho\big[F^{\mu}(\widehat X_0, \varepsilon^0) \tilde Z(X_0,\varepsilon^0) \big]  \Big] \leq 0. 
  $$
 By the arbitrariness of~$ \beta \in  \Ac_1(F^{\mu},\pi)$, this implies~$\E^\rho\big[Z(\varepsilon^0) F^{\mu}(x_0,\varepsilon^0)\big]=0$, $\mu_0(\mathrm{d}x_0)\mbox{--a.s}$, where $ Z(\varepsilon^0):=\E^{\mu_0}[\tilde Z(X_0,\varepsilon^0) ]$ satisfies \eqref{forme_Z}. 

\smallskip 
\noindent (ii)$\Longrightarrow$(i): Under~\ref{th:second}, we obtain for all $\beta \in \Ac_1(F^\mu,\pi)$ that 
$$
\E^\rho\big[ Z(\varepsilon^0) G_{ \mu}^{ \beta}(x_0,\varepsilon^0)\big]=\widehat \E^{\mu_0}\Big[ \beta(x_0,\widehat X_0)\E^\rho\big[Z(\varepsilon^0)  F^{\mu}(\widehat X_0,\varepsilon^0)\big] \Big]=0, \, \, \mu_0(\mathrm{d}x_0)\mbox{--a.s}.
$$
by the Fubini Theorem and the fact that $Z$ is $\varepsilon^0$--measurable. As $Z>0,$ a.s. we deduce that whenever $G_{ \mu}^{ \beta}\geq 0$ a.s, then $G_{\mu}^{ \beta}=0$, a.s. Therefore,  NA holds~\ref{th:first}.
 
\smallskip 
We finally show that (ii) is equivalent to \eqref{condition_coli} for an equilibrium solution $(\mu,\pi)$. As $\varepsilon$ is centered and independent of $(\eps^0,\eta)$, we obtain by taking conditional expectation on $(X_0, \varepsilon^0)$ in  \eqref{1periodMF1}: 
\begin{equation} \frac{F^\mu(X_0,\varepsilon^0)}{m^{\pi}(X_0)}= b(X_0)+\sigma^0(X_0)\varepsilon^0+\widehat \E^{\mu_0}\big[ \pi(X_0,\widehat X_0) F^{\mu}(\widehat X_0, \varepsilon^0)\big], \, \mu_0 \otimes \rho\mbox{--a.s}. \label{F_eta}\end{equation}
which in turn implies by Fubini's Theorem that $\mu_0\mbox{--a.s}$
$$ 
\frac{\E^\rho\big[ Z  F^{\mu}(X_0, \varepsilon^0)\big]}{m^{\pi}(X_0)}
= 
\E^\rho[Z]b(X_0)+ \E^\rho\big[Z \varepsilon^0 \big]\sigma^0(X_0)  +\widehat \E^{\mu_0}\big[ \pi(X_0,\widehat X_0)\E^\rho[ Z F^\mu(X_0,\varepsilon^0)]\big]. 
$$
This provides the required equivalence.
\ep

\subsection{The mean--variance criterion}
\label{one_mean_field_period}

In this section, we specialize the discussion to the case where the performance of the representative agent is measured by the mean--variance criterion:
\be
\label{mean_variance} 
\sup_{\beta\in\Ac_2(F^\mu,\pi)} 
\mathbb{MV}_q\big[X_1^{\mu,\pi,\beta}|X_0\big], 
&\mbox{where}&
\mathbb{MV}_q[\cdot |X_0] 
:=
\E[\cdot |X_0]-\frac{1}{2q}\V[\cdot |X_0], 
\ee 
for some given parameter $q>0$. Although this criterion fails to be nondecreasing, we shall focus on this standard mean--variance problem for computational tractability. Moreover, as standard in the portfolio optimization literature, the mean--variance criterion is often believed to capture the main features of a risk--averse expected utility performance criterion. 

We emphasize that due to the failure of the monotonicity of the mean--variance performance criterion, it is not anymore clear that the no--arbitrage condition is a necessary condition of equilibrium. For this reason, we search for an equilibrium solution of the MFG problem of cross--holding without imposing the NA constraint on the coefficient. Remarkably, our results  below say that such an equilibrium fails to exist if the coefficients $b$ and $\sigma^0$ are not collinear. 

Our starting point is the following equation for $F^\mu$ which follows from taking expectations on both sides of \eqref{1periodMF1}:
\begin{equation}\label{eq_Nash1}
\big(1\!+\!\widehat\E^{\mu_0}[\pi(\widehat X_0,\!X_0)]\big)
F^\mu(X_0,\!\varepsilon^0)
= 
b(X_0)
\!+\!
\sigma^0(X_0)\varepsilon^0+\widehat \E^{\mu_0}\big[ \pi(X_0,\!\widehat X_0) F^{\mu}(\widehat X_0, \varepsilon^0)\big]. 
\end{equation}
This equation expresses  the map $F^\mu$ as a solution of a Fredholm integral equation of the second type whose Kernel depends on the equilibrium strategy $\pi$ and the corresponding environment distribution $\mu$. Our main result provides a complete characterization of equilibrium solutions of our MFG generated by cross--holding strategies under which uniqueness holds for the last Fredholm integral equation. We notice that, under this uniqueness condition, it follows that
\be\label{FG}
F^\mu(X_0, \varepsilon^0)
&:=&
\E^\rho[F^\mu(X_0, \varepsilon^0)]
+\varepsilon^0 F_1^\mu(X_0),
\ee
where $\E^\rho[F^\mu(X_0, \varepsilon^0)]$ satisfies a similar Fredholm equation, by taking expectations conditional to $X_0$, and $F_1$ satisfies
$$
\big(1+\widehat\E^{\mu_0}[\pi(\widehat X_0,X_0)]\big)
F_1^\mu(X_0)
=
\sigma^0(X_0)
+\widehat \E^{\mu_0}\big[ \pi(X_0,\widehat X_0) F_1^\mu(\widehat X_0)\big], 
\mu_0\!\otimes\!\rho\mbox{--a.s.}
$$

\begin{theorem}\label{Theorem_eq_Nash1}
There exists a mean field equilibrium $(F^\mu, \pi) \in \L^2(\mu_0 \otimes \rho) \times \L^2(\mu_0^{\otimes 2}) $ with unique solution to the Fredholm equation \eqref{eq_Nash1} if and only if 
    \b*
    \E[b(X_0)] >0
    &\mbox{and}&
    \sigma^0=\lambda b
    ~~\mbox{with}~~\lambda=\pm\sqrt{\frac{q}{\E[b(X_0)]}},
    \e*
and the strategy $\pi \in \L^2(\mu_0\!\otimes\!\mu_0)$ induces uniqueness for the Fredholm equation \eqref{eq_Nash1} and satisfies 
\be\label{drift_b}
0 
\ne 
\widehat\E^{\mu_0}\big[\pi (X_0, \widehat X_0)\big]
+\frac{b(X_0)}{\E[b(X_0)]}
&=& 
1+\widehat\E^{\mu_0}\big[\pi (\widehat X_0,X_0)\big], \, \, \mu_0\mbox{--a.s.}
\ee
In this case, $F^\mu$ is explicitly given by 
\begin{equation}   
F^\mu(X_0, \varepsilon^0)
=
\E[b(X_0)] ( 1+\lambda \varepsilon^0), \, \, \mu_0 \otimes \rho\mbox{--a.s}.\label{forme_F_var}
\end{equation}
\end{theorem}

The existence of a solution of \eqref{drift_b} satisfying uniqueness for the Fredholm equation \eqref{eq_Nash1} will be discussed in Examples \ref{exp:psi(hatx)} and \ref{exp:phi(x)psi(hatx)} below.
We emphasize again that the last equilibrium characterization imposes a collinearity condition between the drift coefficient of the idiosyncratic risk and its common noise volatility. In particular the mean--variance equilibrium solution of the MFG satisfies the NA condition. 

The equilibrium dynamics derived in the last result are defined by
\begin{equation}
X^{\mu,\pi,\pi}_1
=
X_0
+\E[ b(X_0)](1+\lambda \varepsilon^0 )
+ \frac{\sigma(X_0)\varepsilon}
          {1+\widehat\E^{\mu_0}\big[\pi (\widehat X_0,X_0)\big]}.
\label{SDE_equilibre}
\end{equation}
In particular, the common noise volatility under equilibrium is dilated by $\sqrt{q}$. Moreover, when the mean--variance criterion has large penalty on the variance (i.e. small $q$), the equilibrium common noise is reduced by the factor $\sqrt{q}$. 

We also notice that equation \eqref{drift_b} can be written as  
 \begin{equation*} 
 Y(X_0)
 :=
 \widehat \E^{\mu_0}\left[\pi(X_0, \widehat X_0)\right]-\widehat \E^{\mu_0}\left[\pi( \widehat X_0, X_0)\right]
 =
 1-\frac{b(X_0)}{\E[ b(X_0)]}, \, \, \mu_0\mbox{--a.s,}
 \end{equation*}
Here $Y$ indicates the representative agent's net detention level. As $\E[ b(X_0)]>0$ at equilibrium,  the last expression shows that $Y$ satisfies a mean reversion to the origin driven by the position of drift $b(X_0)$ to the average population drift $\E[ b(X_0)]$.  

\medskip 
For arbitrary strategies $\pi \in \L^2(\mu^{\otimes 2}_0)$, the uniqueness issue of the Fredholm equation  \eqref{eq_Nash1} is not granted, in general. Our restriction to the equilibria generating uniqueness for the Fredholm integral equation imposes more constraints on the set of admissible strategies.

\medskip 
We end this section by the following simplest example of equilibrium strategies satisfying \eqref{drift_b}. 

\begin{example}[cross--holding depending on the second argument only]\label{exp:psi(hatx)} By direct substitution in \eqref{drift_b}, we may find equilibrium strategies $\pi(x,\hat x)$ depending on the $\hat x-$argument only, defined up to a constant $c$ by:
\begin{equation}
\pi(x,\hat x)=\frac{b(\hat x)}{\E[ b(X_0)]}-c 
\mbox{ such that } 
(c-1)\E[b(X_0)] \notin b\big(\mbox{supp}(\mu_0)\big). 
\label{solution_part1}
\end{equation}
Under such equilibrium strategies, the representative agent holds more shares from those competitors with larger expected return. 

We next examine the uniqueness for Freholm equation \eqref{eq_Nash1} induced with such strategies. As $\pi$ depends only on the second argument, equation \eqref{eq_Nash1} becomes 
\b*
F^\mu(X_0,\varepsilon^0)
= 
\frac{\varphi(X_0,\eps^0)+\alpha(\eps^0)}{1\!+ \pi(X_0)},
&\mbox{where}&
\varphi(X_0,\eps^0):=b(X_0)
\!+\!
\sigma^0(X_0)\varepsilon^0
,~\mbox{a.s.} 
\e*
and $\alpha(\eps^0):=\widehat \E^{\mu_0}\big[ \pi(\widehat X_0) F^{\mu}(\widehat X_0, \varepsilon^0)\big]$. Thus, uniqueness for $F^\mu$ reduces to the uniqueness of the map $\alpha(\eps^0)$. Substituting the expression of $F^\mu$, we see that
\b*
\alpha(\eps^0)
=
\frac{\widehat\E\big[\varphi(\widehat X_0,\eps^0)\frac{b(\widehat X_0)-c\E[b(X_0)]}{b(\widehat X_0)-(c-1)\E[b(X_0)]}\big]}
       {\widehat\E\big[\frac{\E[b(X_0)]}{b(\widehat X_0)-(c-1)\E[b(X_0)]}\big]},
\e*
provided that the constant $c$ in \eqref{solution_part1} is chosen so that the denominator of the last expression does not vanish. In particular, notice that $\alpha(\eps^0)$ is affine in $\eps^0$ in agreement with \eqref{FG}.
\end{example}

\begin{example}[Separable form]\label{exp:phi(x)psi(hatx)}  We may also search for solutions of \eqref{drift_b} with separable form $\pi(x,\hat x)=\psi(x)\phi(\hat x)$. Direct substitution provides:
\begin{align}
\pi(x,\hat x)\!=\!\psi(x)\big( c\psi(\hat x)+\frac{b(\hat x)}{\E[ b(X_0) ]}-1 \big) 
\mbox{ with }
& 
\E[ \psi(X_0) ]\!=\!1 \mbox{ and } 
\nonumber\\ &
c\psi(X_0)\!\ne\! \frac{b(X_0)}{\E[ b(X_0) ]},\, \mbox{a.s.}  
\label{exp:phipsi}
\end{align}
We next examine the uniqueness for the Fredholm equation \eqref{eq_Nash1}. As $1+\widehat\E[\pi(\widehat X_0,X_0)=c\psi(X_0)+\frac{b(X_0)}{\E[b(X_0)]}$, it follows that \eqref{eq_Nash1} reduces to
$$
F^\mu(X_0,\eps^0)
=
\frac{\varphi(X_0,\eps^0)+\alpha(\eps^0)\psi(X_0)}
       {c\psi(X_0)+\frac{b(X_0)}{\E[b(X_0)]}}
~\mbox{with}~
\varphi(X_0,\eps^0):=b(X_0)
\!+\!
\sigma^0(X_0)\varepsilon^0
,~\mbox{a.s.}
$$
and $\alpha(\eps^0):=\widehat\E^{\mu_0}\big[ (c\psi(X_0)+\frac{b(X_0)}{\E[b(X_0)]}-1)F^{\mu}(\widehat X_0, \varepsilon^0)\big]$. As in the previous example, the uniqueness for the Fredholm equation is reduced to the uniqueness of the map $\alpha(\eps^0)$, and we obtain by direct substitution that
$$
\alpha(\eps^0)
=
\frac{\widehat\E^{\mu_0}\left[\varphi(X_0,\eps^0)(1-\frac{1}{c\psi(X_0)+\frac{b(X_0)}{\E[b(X_0)]}})\right]}
       {\widehat\E^{\mu_0}\left[\frac{\psi(X_0)}{c\psi(X_0)+\frac{b(X_0)}{\E[b(X_0)]}} \right]},
$$
provided that the constant $c$ in \eqref{exp:phipsi} is chosen so that the denominator in the last expression does not vanish. We notice again  that $\alpha(\eps^0)$ is affine in $\eps^0$ in agreement with \eqref{FG}.
\end{example}

\section{Continuous time MFG of cross--holding} 
\label{sec:set-resutls}

Let $(\Omega,\F,\P)$ be a probability space supporting an $\R^2$--valued Brownian motion $(W,W^0)$. We denote by $\F^0:=\left( \Fc^0_t\right)_{t \in [0,T]}$ the $\P$--completion of the canonical filtration generated by $W^0$ .

\medskip
For some Polish space $E$, we denote by $\Pc_2(E)$ the collection of all probability measures, with finite second moment. Throughout this paper, $E$ will be either $\R^d$, for some integer $d$, or the set $\Cc$ of all continuous maps from $[0,T]$ to $\R$. The canonical process on $\Cc$ is $\widehat{X}:\Cc\longrightarrow\R$ defined by $\widehat{X}_t(\widehat\omega):=\widehat\omega(t)$ for all $t\in[0,T]$ and $\widehat\omega\in\Cc$.

\medskip
For fixed $\mu_0 \in \Pc_2(\R)$, we denote  
\begin{itemize}
\item by $\Sc_2$ the collection of all $\F$--adapted It\^o processes $S$ with initial law $\Lc(S_0)=\mu_0$, and satisfying 
$$
\mathrm{d}S_t
=
B_t(S_t) \drm t
+\Sigma_t(S_t) \drm W_t
+\Sigma^0_t(S_t) \drm W^0_t,
 \;\P\mbox{--a.s.} 
$$
with $\Lc\left(S_{t \wedge \cdot}| \Fc^0_T \right)=\Lc\left(S_{t \wedge \cdot}| \Fc^0_t\right)\;\drm t \otimes\drm\P-$a.e. for some coefficients 
\begin{equation}\label{coef-pm}
B,\Sigma,\Sigma^0:
[0,T] \!\x\! \Om \!\x\! \R\longrightarrow\R,
~\{\Fc_t^0\otimes \Bc(\R)\}_{t \in [0,T]}-\mbox{prog. measurable,}
\end{equation}
and satisfying the integrability condition:
$$
\E \left[ \int_0^T |B_t(S_t)|^2 + |\Sigma_t(S_t)|^2 + |\Sigma^0_t(S_t)|^2 \;\drm t \right]< \infty,$$
in particular, we have $\E \big[ \sup_{t \in [0,T]} |S_t|^2 \big] < \infty$;
\item and by $\Pc^{\F^0}_2(\Cc)$ the collection of all $\Fc^0_T-$measurable $\Pc_2(\Cc)-$valued random variables $\mu$ such that $\mu=\Lc\big(S| \Fc^0_T \big)$ for some $S\in\Sc_2$. 

For such a $\mu\in\Pc^{\Fc^0}_2(\Cc)$, we shall denote $\mu_t:=\Lc(S_t|{\Fc^0_t})=\mu \circ (\Xh_t)^{-1}$, and $\widehat\E^\mu[\phi(\widehat X)]:=\int\phi(\widehat \omega)\mu(\drm\widehat \omega)$, for all map $\phi:\Omega\times\Cc\longrightarrow\R$ with $\mu-$integrable $\phi(\omega,.)$, $\P-$a.e. $\omega$.
\end{itemize}

\medskip
We next introduce the set $\Ac$ of scalar $\left(\Fc^0_t \otimes \Bc_{\R^2} \right)_{t \in [0,T]}$--progressively measurable maps from $[0,T] \x \Om \x \R^2$ to $\R$. For all $\mu\in\Pc^{\Fc^0}_2(\Cc) $ and $\pi \in \Ac$, we denote
\b*
m_t^{\mu,\pi}(x)
:=
 \frac{1}{1+\int_{\R} \pi_t(\hat x,x)\mu_t(\drm\hat x)},
&t\in[0,T],&x\in\R,
\e*
with the convention $\frac{1}{0}=\infty$  and we define by $\Ac^{\mu,\pi}_b$ the subset of $\Ac$ consisting of bounded $\beta \in \Ac$ such that the map from $\R$ to $\R\times\R$:
\b*
x \longmapsto m^{\mu,\pi}_t(x)\widehat{\E}^{\mu} \left[ \beta_t(x,\Xh_t) (B_t, \Sigma^0_t)(\Xh_t) \right] 
&\mbox{is continuous,}&
\mathrm{d}t \otimes \mathrm{d}\P-\mbox{a.e.}
\e*
We further define appropriate conditions so that, given $\mu\in \Pc^{\F^0}_2(\Cc)$ and $\pi\in\Ac$, the following cross--holding state dynamics induce a well defined process started from an $\Fc_0-$r.v. $X_0$ with law $\mu_0$:
\begin{equation}\label{SDE:X}
X_.
=
X_0+P_\cdot
+
\widehat\E^{\mu}\left[\int_0^\cdot  \beta_s(X_s,\widehat{X}_s) \drm\widehat{X}_s \right]
-
\int_0^\cdot  \widehat\E^{\mu}\left[\pi_s(\widehat{X}_s,{X}_s) \right]\drm{X}_s
,\;\;\P\mbox{--a.s.}
\end{equation}
with idiosyncratic risk process $P$ defined by
\begin{equation}\label{SDE:P}
P_\cdot
:=
\int_0^\cdot b_s(X_s,\mu_s) \drm s
+
\int_0^\cdot \sigma_s(X_s,\mu_s) \drm W_s
+
\int_0^\cdot \sigma^0_s(X_s,\mu_s) \drm W^0_s,\;\;\P\mbox{--a.s.}
\end{equation}
for some $\F^0\otimes\Bc_{\R}\otimes\Bc_{\Pc_2(\R)}-$progressively measurable coefficients
$$
b,\sigma,\sigma^0:[0,T]\times\Omega\times\R\times\Pc_2(\R)
\longrightarrow
\R
$$
satisfying the following conditions ensuring the existence of a unique square integrable solution for the McKean-Vlasov SDE \eqref{SDE:X} with no interaction $\beta=0$.

\begin{assumption}
For $\varphi=(b,\sigma,\sigma^0)$, the map $\varphi:[0,T]\times\Omega\times\R \times\Pc_2(\R) \longrightarrow \R$ is

{\rm (i)} Lipschitz in $(x,m)$ uniformly in $(t,\om)$, 

{\rm (ii)} and has uniform linear growth in $(x,m)$ in the following sense:
\begin{align*}
    \sup_{(t,\;\om,\;x,\;m)}\frac{|\varphi(t,\om,x,m)|}{1+|x| + \left(\int_\R x^2\;m(\mathrm{d}x) \right)^{1/2}} < \infty.
\end{align*}

\end{assumption}

\noindent The dynamics of the state process $X$ in \eqref{SDE:X} have the following interpretation:
\begin{itemize}
\item The probability measure $\mu$ is the $W^0-$conditional law of the process $\Xh$ which represents the population surrounding the representative agent with state process $X$;
\item The strategy $\pi$ represents the investment of the population in the equity of the representative agent;
\item  In our definition below of the MFG, $\Xh$ can be seen as a conditional independent copy of the equilibrium dynamics of $X$, with distribution $\mu$, and $\pi$ is an equilibrium cross--holding strategy of the surrounding population. 
\item For this reason, we say that the pair $(\mu,\pi)$ is a {\it random environment}, meaning that it describes the population of equity processes surrounding the representative agent. We then introduce the set of random environments
$$
\Ec
:=
\big\{(\mu,\pi): \mu\in\Pc_2^{\F^0}(\Cc)\;\mbox{and}\;\pi\in\Ac
\big\}.
$$
\item The strategy $\beta$ represents the investment strategy that the representative agent may implement in reaction to the random environment $(\mu,\pi)$. We naturally define below the solution of the MFG as a random environment under which the representative agent finds no benefit to deviate from the surrounding population's cross--holding strategy, and the induced state process reproduces the law of the surrounding population equity process.
\end{itemize}
In the following definition, $(B,\Sigma,\Sigma^0)$ is the triplet of coefficients corresponding to the above $\mu$, and we denote for all $\beta\in\Ac$
\be\label{barBSigma}
    \left(\!\!\begin{array}{c}
           \overline{B}^{\mu,\pi,\beta}_t
           \\
           \overline{\Sigma^0}^{\mu,\pi,\beta}_t  
           \end{array}
   \!\!\right)\!(x)
&:=&
    \int_{\R}\!\! \beta_t(x,\widehat x) \!
                      \left(\!\! \begin{array}{c}
                             B_t
                             \\
                             \Sigma^0_t
                             \end{array}
                     \!\!\right)\!(\widehat x) \mu_t(\drm\widehat x) 
                     +
                     \left(\!\!\begin{array}{c}
                            b_t
                            \\
                            \sigma^0_t 
                            \end{array}
                    \!\!\right)\!(x,\mu_t).
\ee 

\begin{definition}\label{def:admissible}
{\rm (1)} A random environment $(\mu,\pi)\in\Ec$ is admissible if 
\begin{itemize}
\item[{\rm (1-i)}] $\drm t \otimes \drm\P\mbox{--a.e.}$, $0\neq 1+\int_{\R} \pi_t(\widehat x,x)\mu_t(\drm\widehat x)=\frac{1}{m_t^{\mu,\pi}(x)}$ for all $x\in\R.$

\item[{\rm (1-ii)}] The SDE defined by the starting data $X_0$ with law $\mu_0$ and
\be \label{eq:common_main}
    \drm X_t
    &\hspace{-2mm}=&\hspace{-2mm}
    m_t^{\mu,\pi}(X_t) 
    \left(\overline{B}^{\mu,\pi,\beta}_t(X_t)  \drm t
           \!+\! \overline{\Sigma^0}^{\mu,\pi,\beta}_t(X_t) \drm W^0_t
           \!+\!\sigma_t(X_t,\mu_t) \drm W_t 
   \right)
\ee
has a weak solution, for all cross--holding strategy $\beta\in\Ac^{\mu,\pi}_b$, and satisfies  
\be \label{eq:estimates_common0}
\widehat\E^{\mu} 
\Big[\int_0^T \beta_t (X_t,\Xh_t)^2 
                     \big(|B_t|^2+ |\Sigma_t|^2+ |\Sigma^0_t|^2 \big)(\Xh_t)\mathrm{d}t \Big] < \infty,
&\P\mbox{--a.e.}&
\ee
and
\be \label{eq:estimates_common}
\E \Big[ \int_0^T m_t^{\mu,\pi}(X_t)^2 
              \big( \overline{B}^{\mu,\pi}_t(X_t)^2 
                      + \overline{\Sigma^0}^{\mu,\pi}_t(X_t)^2 
                      + \sigma_t(X_t,\mu_t)^2
              \big)\mathrm{d}t 
    \Big] 
&<& 
\infty.
\ee

\end{itemize}
{\rm (2)} A strategy $\beta\in\Ac$ is said to be $(\mu,\pi)-$admissible, and we write $\beta\in\Ac(\mu,\pi)$, if it satisfies the above \eqref{eq:common_main}, \eqref{eq:estimates_common0} and \eqref{eq:estimates_common}.
\end{definition}

Notice that $\Ac^{\mu,\pi}_b\subset\Ac(\mu,\pi)$, and that the dynamics of $X$ introduced in the last definition exactly reproduces the required cross--holding dynamics in \eqref{SDE:X}-\eqref{SDE:P}.
We also observe that, although for $\P$--a.e. $\om \in \Om$, the canonical process $\Xh$ is not a semi--martingale under $\mu(\om,d\hat\omega)$, the quantity $\E^{\mu}\big[\int_0^\cdot  \beta_s(X_s,\widehat{X}_s) \drm\widehat{X}_s \big]$ is always well defined $\P$--a.e. 

\begin{example} Let $\mu\in\Pc^{\F^0}_2(\Cc) $ and $\pi \in \Ac$ be such that $x\longmapsto m^{\mu,\pi}_t(x)$ is continuous $\drm t \!\otimes\! \mathrm{d}\P-$a.e. and
$$
\E\widehat\E^{\mu} 
\Big[ \int_0^T  \sup_{x \in \R}|m^{\mu,\pi}_t(x)|^2 \big( |B_t(\Xh_t)|^2 + |\Sigma_t(\Xh_t)|^2 + |\Sigma^0_t(\Xh_t)|^2 \big)\;\drm t \Big]<\infty,
$$
Then, we can verify that $(\mu,\pi)$ is admissible.
\end{example}

For $\beta\in\Ac(\mu,\pi)$, we denote by $X^{\mu,\pi,\beta}$ an arbitrary solution of the SDE \eqref{eq:common_main}, and we introduce the reward function
\begin{align*}
    J_{\mu,\pi}(\beta)
    :=
    \E \left[ U (X^{\mu,\pi,\beta}_T)\right],
\end{align*}
where  $U:\R\longrightarrow\R$ is a non--decreasing utility function verifying $\lim_{x \to \infty} U(x)=\infty$.

\begin{definition}\label{def:NE}
An admissible random environment $(\mu,\pi) \in \Pc_2^{\F^0}(\Cc) \x \Ac$ is an equilibrium solution of the {\rm MFG} of cross--holding if: 

\medskip
{\rm (i)} $J_{\mu,\pi}(\beta)\le J_{\mu,\pi}(\pi)
        < \infty$, for all $\beta \in \Ac(\mu,\pi)$,

\medskip
{\rm (ii)} $\pi \in \Ac(\mu,\pi)$, $\Lc(X^{\pi,\pi}|{\Fc^0_T})
        =
        \mu$, $\P$--a.e. and $\mathrm{d}t\otimes\mu_t(\drm x)-$a.e.
$$
\Sigma_t(x)=m_t^{\mu,\pi}(x)\sigma_t(x),
~B_t(x)=m_t^{\mu,\pi}(x) \overline{B}_t^{\mu,\pi,\beta}(x),
~\Sigma^0_t(x)
=
m_t^{\mu,\pi}(x) \overline{\Sigma^0}_t^{\mu,\pi,\beta}(x),
$$
with the notations of \eqref{barBSigma}.
 \end{definition}

The last equation follows from the identification of the equilibrium dynamics with the coefficients $(B,\Sigma,\Sigma^0)$ of the It\^o process $X^{\mu,\pi,\pi}:=X$ under $\mu$: 
     \begin{align*}
         X_t
         &=
         X_0
         \!+\!\int_0^tm^{\mu,\pi}(X_s)\drm P_s
         \!+\!
         \widehat\E^{\mu}\Big[\int_0^t  m^{\mu,\pi}(X_s)
                                           \pi_s(X_s,\widehat{X}_s) \drm\widehat{X}_s 
                      \Big]
\\       &=
         X_0
         \!+\!\int_0^tm^{\mu,\pi}(X_s)\drm P_s
         \!+\!
         \widehat\E^{\mu}\Big[\int_0^t\!\!  m^{\mu,\pi}(X_s)
                                           \pi_s(X_s,\!\widehat{X}_s) 
                                           \big(\!B_s(\widehat X_s)\drm s 
                                                  \!+\!\Sigma^0_s(\widehat X_s)\drm W^0_s
                                           \big)
                      \Big].
     \end{align*}
     
\section{Arbitrage free random environment}

In this section, we introduce the notion of no--increasing profit, a weaker concept of no--arbitrage introduced in the previous literature, see Fontana \cite{Fontana}. Similar to the standard no--arbitrage condition, the no--increasing profit condition is necessary for the existence of equilibrium in the context where the performance criterion is increasing in terms of the state process $X$. We shall denote by
\b*
G^{\mu,\pi}_t(\beta)
:=
\widehat\E^{\mu}\Big[\int_0^t\beta_s(X^{\mu,\pi,\beta}_s,\widehat{X}_s) 
                                              m^{\mu,\pi}(X^{\mu,\pi,\beta}_s)\drm\widehat{X}_s 
                           \Big],
&\mbox{for all}&
\beta\in\Ac(\mu,\pi).
\e*

\begin{definition}
An admissible random environment $(\mu,\pi) \in \Pc_2^{\F^0}(\Cc) \x  \Ac$ is said to have no increasing profit {\rm (NIP)} if for all strategy $\beta\in\Ac^{\mu,\pi}_b$ with $t\longmapsto G^{\mu,\pi}_t(\beta)$ nondecreasing, $\P$--a.s. we have $G^{\mu,\pi}_T(\beta)=0$, $\P-$a.s.
\end{definition}

In order to justify this definition, notice that if the admissible random environment $(\mu,\pi)$ is not free of no--increasing profit, then we may find some bounded $\beta\in\Ac^{\mu,\pi}_b$ violating the condition of the last definition. Then, it follows from Definition \ref{def:admissible} (1-ii) that $\lambda\beta\in\Ac(\mu,\pi)$ for all $\lambda>0$, and this in turn implies that $\P\big[\lim_{\lambda\to\infty}X^{\lambda\beta,\mu,\pi}_T=\infty\big]>0$. Therefore, the existence of an optimal response (i.e. a maximizer) in item (i) of Definition \eqref{def:NE} fails to hold under any increasing agent's performance criterion.

\medskip
The main result of this section states that the coefficients $(B,\Sigma^0)$ of an increasing profit free admissible random environment satisfy a proportionality condition similar to the Heath--Jarrow--Morton \cite{heath1992bond} restriction in the term structure modeling. 

 \begin{theorem} \label{thm:equivalence_NIP}
An admissible random environment $(\mu,\pi)\in \Pc_2^{\F^0}(\Cc) \x  \Ac$ satisfies the {\rm (NIP)} condition if and only if the corresponding coefficients $B$ and $\Sigma^0$ satisfy the following proportionality condition:
$$
B_t(x)
=
\Sigma^0_t(x)\lambda_t ,
\;\drm t\otimes\mu_t(\drm x)\mbox{--a.e.}
\;\;\P\mbox{--a.s.}
$$ 
for some scalar $\G-$progressively measurable process $(\lambda_t)_{t \in [0,T]}$.
\\
If in addition $(\mu,\pi)$ is an equilibrium solution of the {\rm MFG} of cross--holding, then 
    \begin{align*}
        b_t(x,\mu_t)=\lambda_t  \sigma^0_t(x,\mu_t),\;\;\drm t\otimes\mu_t(\drm x)\mbox{--a.e.}
\;\;\P\mbox{--a.s.}
    \end{align*}
 \end{theorem}

\proof {\bf (i)} We start with the last part of the statement, given the characterization of the NIP condition which we prove in (ii)-(iii) below. Since $(\mu,\pi)$ is an equilibrium solution of the {\rm MFG} of cross--holding, we know that  $\P$--a.e., 
$\mathrm{d}t\otimes\mu_t(\drm x)-$a.e.
\be\label{coefEq}
\Big(\!\!\begin{array}{c}B_t\\ \Sigma^0_t\end{array}\!\!\Big)(x) 
=
m_t^{\mu,\pi}(x) 
\Big( \int_{\R} \pi_t(x,\widehat x) 
                      \Big(\!\!\begin{array}{c}B_t\\ \Sigma^0_t\end{array}\!\!\Big)
                      (\widehat x)\mu_t(\drm\widehat x) 
        + \Big(\!\!\begin{array}{c}b_t\\ \sigma^0_t\end{array}\!\!\Big)(x,\mu_t)  
\Big).
\ee
Using the characterization $B
=
\Sigma^0\lambda$, $\drm t\otimes\mu_t(\drm x)-$a.e., $\P-$a.s. of the NIP condition stated in the first part, it follows from the first equation in \eqref{coefEq} that
\begin{align*}
    m^{\mu,\pi}_t(x) b_t(x,\mu_t)
    &= B_t(x) - m^{\mu,\pi}_t(x) \int_\R \pi_t(x,\widehat x) B_t(\widehat x) \mu_t(\mathrm{d}\widehat x)
    \\
    &= \lambda_t \left( \Sigma^0_t(x) - m^{\mu,\pi}_t(x) \int_\R \pi_t(x,\widehat x) \Sigma^0_t(\widehat x) \mu_t(\mathrm{d}\widehat x) \right)
    =
    \lambda_t m^{\mu,\pi}_t(x) \sigma^0_t(x,\mu_t),
\end{align*}
by the second equation in \eqref{coefEq}. The required result follows from the fact hat $m^{\mu,\pi}_t(x) \neq 0$, due to the admissibility of the random environment $(\mu,\pi)$.

\medskip
\noindent {\bf (ii)} We next prove the necessary condition for the characterization of the NIP condition. Let $c >0$ and $N \ge 1$, and set 
$$
    H^{c,N}_t(\lambda):= \widehat{\E}^{\mu} \left[ \left|B_t(\Xh_t)- \lambda \Sigma^0_t(\Xh_t) \right|^2 \1_{A^{c,N}_t}(\Xh_t)\right] 
$$   
where
$$
    A^{c,N}_t:=\left\{ x:\; |B_t(x)|^2 + |\Sigma^0_t(x)|^2 \le N, \;\;c \le |\Sigma^0_t(x)| \right\}.
$$
Then, the $\F^0-$progressively measurable process 
$$
\lambda^{c,N}_t
:=
\frac{\widehat \E^\mu\left[\Sigma^0_t(\widehat X_t)B_t(\widehat X_t) \1_{A^{c,N}_t}(\Xh_t) \right]}
       {\widehat \E^\mu\left[\Sigma^0_t(\widehat X_t)^2 \1_{A^{c,N}_t}(\Xh_t)\right]},
\;\;t\in[0,T],
$$
is a minimizer of $H^{c,N}_t$ over the set of $\F^0-$progressively measurable processes $\lambda$. We next introduce the scalar $\left(\Fc^0_t \otimes \Bc(\R) \right)_{t \in [0,T]}$--progressively measurable  process 
\be\label{almostprop}
\nu^{c,N}_t(\widehat X_t)
:=
\left(B_t(\widehat X_t)-\Sigma^0_t(\widehat X_t)\lambda^{c,N}_t \right) \1_{A^{c,N}_t}(\Xh_t),
\ee
which satisfies
\begin{equation}\label{nu}
\begin{array}{c}
\widehat \E^\mu\big[\Sigma^0_t(\widehat X_t)\1_{A^{c,N}_t}(\Xh_t)\hat\nu^{c,N}_t(\Xh_t)\big]=0,
~~\P-\mbox{a.s.}
\;\;t\in[0,T].
\end{array}
\end{equation}
We now claim that the {\rm NIP} condition implies that 
\be\label{claim}
\hat\nu^{c,N} &=& 0.
\ee
Before proving this claim, let us show how it induces the required necessary condition. Since $\mu \in \Pc^{\F^0}_2(\Cc)$, we have $\1_{A^{c,N}_t}(x)\longrightarrow 1$, as $N \nearrow \infty$ and $c \searrow 0$, $\mathrm{d}t \otimes \mu_t(\mathrm{d}x)\;\mbox{--a.e.}$ Moreover, it follows from \eqref{claim} and \eqref{almostprop} that the process $\lambda_t:=\limsup_{N \to \infty,\; c \to 0}\lambda^{c,N}_t$, satisfies the required proportionality condition $B=\Sigma^0\lambda$, $\drm t\otimes\mu_t(\drm x)-$a.e., $\drm\P-$a.s.

To complete the proof of the necessary conditions, we prove \eqref{claim} in the two following sub--steps.

\begin{itemize}
\item[{\bf (a)}] We first show that the {\rm NIP} condition implies that for all bounded $\R$--valued $\left(\Fc^0_t \otimes \Bc(\R) \right)_{t \in [0,T]}$--progressively measurable process $\beta$:
\begin{equation}\label{ac}
\widehat \E^\mu\big[(\beta_t\Sigma^0_t)(\widehat X_t)\big]
\!=\!
0,
~\drm t\otimes\drm\P\!-\!\mbox{a.e.}
~\mbox{implies}~
\widehat\E^\mu\big[(\beta_tB_t)(\Xh_t)\big]\!=\!0,
\;\drm t\otimes\drm\P\!-\!\mbox{a.e.}
\end{equation}

Indeed, assume to contrary that $\widehat\E^\mu\big[\beta_t(\widehat X_t)\Sigma^0_t(\widehat X_t)\big]=0,$ $\drm t\otimes\drm\P\mbox{--a.e.}$ and   
\b*
\int_0^T \int_{\Om} \1_{A}(t,\om) \P(\mathrm{d}\om) \mathrm{d}t >0
&\mbox{where}&
    A
    :=
    \left\{\widehat\E^{\mu}\left[ \beta_t( \widehat X_t) B_t(\widehat X_t)\right] \neq 0   \right\}.
\e*
for some bounded $\beta$. We define $\beta^*_t(x,\widehat X_t):=s^0_t(x)\beta_t(\widehat X_t)$, with $s^0_t(x):={\rm sgn}\left\{m^{\mu,\pi}_t(x)\widehat\E^\mu[\beta_t(\widehat X_t)B(t, \widehat{X}_t)]\right\}$. Then $\beta^* \in \Ac(\mu,\pi)$ is admissible, and
\begin{align*}
    G^{\mu,\pi}_t(\beta^*) 
    &=
    \widehat\E^{\mu} \left[ \int_0^t m^{\mu,\pi}_s(X^{\mu,\pi,\beta^*}_s)
                                       \beta^*_s(X^{\mu,\pi,\beta^*}_s,\widehat{X}_s) 
                                       \mathrm{d}\widehat{X}_s 
                  \right] 
    \\
    &=
    \widehat\E^{\mu} \left[ \int_0^t m^{\mu,\pi}_s(X^{\mu,\pi,\beta^*}_s)
                                       \beta^*_s(X^{\mu,\pi,\beta^*}_s, \widehat{X}_s) 
                                       B(s, \widehat{X}_s) \mathrm{d}s 
                  \right] 
    \\
    &=
     \int_0^t \left|m^{\mu,\pi}_s(X^{\mu,\pi,\beta^*}_s)
                    \widehat\E^{\mu}\left[\beta_s( \widehat{X}_s) B(s, \widehat{X}_s) 
                                      \right]
                 \right|  \mathrm{d}s,
\end{align*}
Consequently $t\longmapsto G^{\mu,\pi}_t(\beta^\star)$ is non--decreasing, $\P$--a.s. Since $m^{\mu,\pi}_t(X^{\mu,\pi,\beta^*}_t ) \neq 0$ $\mathrm{d}t \otimes \mathrm{d}\P$--a.e., by definition of the set $A$, we deduce that $\P\big[G^{\mu,\pi}_T(\beta^\star)\!>\!0\big] >0$. This is in contradiction with the non--increasing profit condition  {\rm (NIP)}. 

\item[{\bf (b)}] We now show that \eqref{ac} implies that $\hat\nu^{c,N}_t(\Xh_t)=0$. Notice that $\sup_{(t,\om)} \lambda^{c,N}_t(\om) < \infty$. Then, $\left( \nu^{c,N}_t(x) \1_{A^{c,N}_t}(x)\right)_{(t,x) \in [0,T] \x \R}$ is a bounded scalar $\left(\Fc^0_t \otimes \Bc(\R) \right)_{t \in [0,T]}$--prog. meas. process. By \eqref{nu}, this process satisfes the left hand side of \eqref{ac}, and it follows from (a) that the right hand side of \eqref{ac} holds, namely $\widehat \E^\mu\big[B_t(\widehat X_t)\hat\nu^{c,N}_t(\Xh_t) \1_{A^{c,N}_t}(\Xh_t)\big]=0$. By the definition of $\hat\nu^{c,N}$ in \eqref{almostprop} together with \eqref{nu}, we get
\begin{align*}
    \widehat \E^\mu\big[\hat\nu^{c,N}_t(\Xh_t)^2 \1_{A^{c,N}_t}(\Xh_t)\big]=0.
\end{align*}
which provides the the required \eqref{claim}.
\end{itemize}
\noindent {\bf (iii)} We finally prove the sufficient condition for the characterization of the NIP condition. Assume that $B_t(x)
=
\Sigma^0_t(x)\lambda_t ,
\;\drm t\otimes\mu_t(\drm x)-\mbox{a.e.}
$, $\P$--a.e. for some $\F^0-$prog. meas. process $\lambda$.
then the NIP condition is true. Let $\beta \in \Ac(\mu,\pi)$ be such that the process $(G^{\mu,\pi}_t(\beta))_{t \in [0,T]}$ is non--decreasing, and let us show that the process $G^{\mu,\pi}(\beta)$ is zero, $\P-$a.s. 

As a non-increasing process, $G^{\mu,\pi}(\beta)$ has finite variation, implying that its martingale part is zero:
\begin{align*}
    0=R_t
    &:=
     \int_0^t \widehat\E^{\mu} \left[ \beta_s(X^{\mu,\pi,\beta}_s, \widehat{X}_s) m^{\mu,\pi}_s(X^{\mu,\pi,\beta}_s) \Sigma^0_s( \widehat{X}_s) \right]  \mathrm{d}W^0_s
    \\
    &=G_t^{\mu,\pi}(\beta)- \int_0^t \widehat\E^{\mu} \left[\beta_s(X^{\mu,\pi,\beta}_s, \widehat{X}_s) m^{\mu,\pi}_s(X^{\mu,\pi,\beta}_s) B_s( \widehat{X}_s) \right]  \mathrm{d}s,
~~t\in[0,T].
\end{align*}
As $\lambda$ is $\F^0-$progressively measurable, this can be written equivalently as  
\b*
G_t^{\mu,\pi}(\beta)
= 
\int_0^t \lambda_s \widehat\E^{\mu} \left[\beta_s(X^{\mu,\pi,\beta}_s, \widehat{X}_s) m^{\mu,\pi}_s(X^{\mu,\pi,\beta}_s) \Sigma^0_s(\widehat{X}_s) \right]  \mathrm{d}s,
&\mbox{for all}&
t\in[0,T].
\e*
On the other hand, we have that 
$$
\int_0^t \lambda_s \1_{|\lambda^0_s| \le N} \widehat\E^{\mu} \left[\beta_s(X^{\mu,\pi,\beta}_s, \widehat{X}_s) m^{\mu,\pi}_s(X^{\mu,\pi,\beta}_s) \Sigma^0_s(\widehat{X}_s) \right]  \mathrm{d}s
=
\left \langle \int_0^\cdot \lambda_s \1_{|\lambda^0_s| \le N}  \mathrm{d}W^0_s, R   \right\rangle_t =0.
$$
Moreover, as $\beta \in \Ac(\mu,\pi)$ is an admissible strategy, we have 
$$
\int_0^t  \E \widehat\E^{\mu} \left[ \left|\lambda_s\beta_s(X^{\mu,\pi,\beta}_s, \widehat{X}_s) m^{\mu,\pi}_s(X^{\mu,\pi,\beta}_s) \Sigma^0_s(\widehat{X}_s) \right| \right]  \mathrm{d}s
< \infty.
$$
We may then take the limit $N \to \infty$ in the last equality, and deduce from the dominated convergence Theorem that $G_t^{\mu,\pi}(\beta)=0$, $t \in [0,T]$.
\ep

\section{The reduced mean field game}

In this section, we derive an important reduction of the mean field game of cross-holding. Recall that any admissible random environment $(\mu,\pi) \in \Pc_2^{\F^0}(\Cc) \x \Ac$ which is a solution of the mean field game satisfies the No--Increasing Profit condition. By Theorem \ref{thm:equivalence_NIP}, we have the following proportionality relation between the drift and the volatility of the common noise:
\b*
B_t(x)=\lambda_t\Sigma^0_t(x)
&\mbox{and}&
b_t(x)=\lambda_t\sigma^0_t(x),
~\drm t\otimes\mu_t(\drm x)-\mbox{a.e. }
\P-\mbox{a.s.}
\e*
for some $\F^0-$progressively measurable process $(\lambda_t)_{t \in [0,T]}$. Here, for notational simplicity, we have omitted the dependence of the coefficients on the marginal law $\mu_t$ because this is not playing any role. Plugging this into the dynamics of the equity process \eqref{eq:common_main}, we see that 
 \be\label{X:alpha}
    \drm X^{\mu,\pi,\beta}_t
    &\hspace{-2mm}=&\hspace{-2mm}
    A^{\mu,\pi}_t\big(X^{\mu,\pi,\beta}_t,\beta_t\big) 
    \big(\lambda_t\drm t+\drm W^0_t\big)
    +v^{\mu,\pi}_t\big(X^{\mu,\pi,\beta}_t\big) \drm W_t,
\ee
where
\b*
A^{\mu,\pi}_t(x,\beta_t)
:=
m_t^{\mu,\pi}(x)
\overline{\Sigma^0}^{\mu,\pi,\beta}_t(x),
&\mbox{and}&
v^{\mu,\pi}_t(x)
:=
m_t^{\mu,\pi}(x)
\sigma_t(x),
\e*
where we recall the notation of \eqref{barBSigma}:
$$
\overline{\Sigma^0}^{\mu,\pi,\beta}_t(x)
=
\Big[\sigma^0_t(x)
       +\!\int_{\R}\!\beta_t(x,\widehat x)\Sigma^0_t(\widehat x)\mu_t(\drm\widehat x)
\Big].
$$
Motivated by this reduced form, we introduce the reduced performance criterion $\J$ defined by
$$
\J_{\mu,\pi}(\alpha)
:=
\E\big[U\big(\X^{\mu,\pi,\alpha}_T\big)\big],
$$
where the reduced controlled state $\X$ is defined by 
\be\label{SDEdbX}
\drm\X^{\mu,\pi,\alpha}_t
&=&
\alpha_t\big(\X^{\mu,\pi,\alpha}_t\big)
             \big(\lambda_t\drm t + \drm W^0_t\big)
+v^{\mu,\pi}_t\big(\X^{\mu,\pi,\alpha}_t\big)\drm W_t,
\ee
and the reduced control process
\begin{equation*}
\begin{array}{c}
\alpha:[0,T]\times\Omega\times\R\longrightarrow\R,
~\F^0\otimes\Bc_\R-\mbox{prog. meas. such that}
\\
\displaystyle
\mbox{\eqref{SDEdbX} has a strong solution satisfying}~
\end{array}
\end{equation*}
\begin{align}\label{alpha}
    \E\left[\int_0^T\!\!\alpha_t(\X^{\mu,\pi,\alpha}_t)^2(1+|\lambda_t|^2) + m_t^{\mu,\pi}(\X^{\mu,\pi,\alpha}_t)^2 \sigma_t\left(\X^{\mu,\pi,\alpha}_t \right)^2\mathrm{d}t\right]<\infty.
\end{align}
The reduced control $\alpha$ represents the aggregate instantaneous return from the cross--holding strategy $\beta$ representing the investment of the representative agent in the surrounding environment. 

The main consequence of the NIP condition, which we state below, says that the individual optimization step of any equilibrium solution of the mean field game of cross--holding reduces to the problem $\sup_\alpha \J_{\mu,\pi}(\alpha)$. Observe that the last problem is nothing but the standard portfolio optimization problem under random endowment represented by the additional term $v^{\mu,\pi}_t\big(\X^{\mu,\pi,\alpha}_t\big)\drm W_t$. Such problems have been studied extensively in the previous literature, and are well--known to differ from the standard Merton portfolio optimization problem only when the endowment process introduces an additional risk, because otherwise it can be perfectly hedged thus reducing to the classical situation. In contrast, our reduced model is a typical incomplete market situation as the random endowment process is driven by the idiosyncratic noise $W$, while the portfolio optimization acts on the common noise $W^0$. The existing literature is mainly focused on the convex duality representation of the problem, see e.g. Cvitani\'c, Schachermayer \& Wang \cite{CvitanicSchachermayerWang}, Hugonnier \& Kramkov \cite{HugonnierKramkov}, and also extensions by Karatzas \& Zitkovi\'c \cite{KaratzasZitkovic},
Zitkovi\'c \cite{Zitkovic} and Mostovyi \cite{Mostovyi} when the intermediate consumption is allowed. 

\begin{proposition} \label{prop:equivence_MFG}
Let $(\mu,\pi) \in \Pc_2^{\F^0}(\Cc) \x \Ac$ be an admissible random environment with $\mu_t\big[\Sigma^0_t\neq 0\big]>0$, $\P-$a.s. Then for any $\beta \in \Ac(\mu,\pi)$, there is $\alpha$ satisfying \eqref{alpha} such that 
\begin{align} \label{eq:equivalence-diffusion}
\X^{\mu,\pi,\alpha}=X^{\mu,\pi,\beta},\;\;\P\mbox{--a.e.}
\end{align}
Conversely, for any $\alpha$ satisfying \eqref{alpha}, there is $\beta \in \Ac(\mu,\pi)$ such that  \eqref{eq:equivalence-diffusion} holds.

\medskip
In addition, if we set $\alpha^{\mu,\pi}_t(x):=A^{\mu,\pi}_t(x,\pi_t)$ for all $(t,x)\in[0,T]\times\R$, the pair $(\mu,\pi)$ is an equilibrium solution of the {\rm MFG} of cross--holding if and only if: 

\vspace{2mm}
{\rm (i)} $\J_{\mu,\pi}(\alpha)\le \J_{\mu,\pi}(\alpha^{\mu,\pi})
        < \infty$, for all $\alpha$ satisfying \eqref{alpha},

\vspace{2mm}
{\rm (ii)} $\Lc\big(\X^{\mu,\pi,\alpha^{\mu,\pi}}|{\Fc^0_T}\big)
        =
        \mu$, $\P$--a.e.
\end{proposition}

\proof
First, let us observe that for any $\alpha$ satisfying \eqref{alpha}, there is $\beta \in \Ac(\mu,\pi)$ admissible s.t. $\alpha_t\left( \X^{\mu,\pi,\alpha}_t\right)=A^{\mu,\pi}_t\left( \X^{\mu,\pi,\alpha}_t,\beta_t \right)$ $\mathrm{d}t \otimes \mathrm{d}\P$--a.e. Indeed, by setting $\beta_t(x,\hat x)=f_t(x){\rm sg}[\Sigma^0_t(\hat x)]$, with
$$
\frac{\alpha_t(x)}{m_t^{\mu,\pi}(x)}
=
\sigma^0_t(x)+f_t(x)\int_{\R}|\Sigma^0_t(\hat x)|\mu_t(\drm\hat x),
$$
This clearly defines $f_t(x),(t,x)\in[0,T]\times\R$, since $\mu_t\big[\Sigma^0_t\neq 0\big]>0$, $\P-$a.s.
We can see that $\X^{\mu,\pi,\alpha}=X^{\mu,\pi,\beta}$ $\P$--a.e. and check that $\beta \in \Ac(\mu,\pi)$. Conversely, for any $\beta \in \Ac(\mu,\pi)$, by defining $\alpha^{\mu,\pi,\beta}_t\left( x\right):=A^{\mu,\pi}_t\left( x,\beta_t \right)$, we can check that $\alpha^{\mu,\pi,\beta}$ satisfies \eqref{alpha} and $\X^{\mu,\pi,\alpha^{\mu,\pi,\beta}}=X^{\mu,\pi,\beta},$ $\P$--a.e.

\medskip
Let $(\mu,\pi)$ be an equilibrium solution of the {\rm MFG} of cross--holding. Then
\begin{align*}
    \J_{\mu,\pi}(\alpha)
    =
    J_{\mu,\pi}(\beta) \le J_{\mu,\pi}(\pi)= \J_{\mu,\pi}(\alpha^{\mu,\pi})
\end{align*}
and $\Lc\big(\X^{\mu,\pi,\alpha^{\mu,\pi}}|{\Fc^0_T}\big)=\Lc\big(X^{\mu,\pi,\pi}|{\Fc^0_T}\big)=\mu$. The converse equivalence follows by using similar arguments.
\ep

\vspace{3mm}
We conclude this section by isolating a simple result which allows to recover a cross--holding strategy from  some given aggregate hedging strategy $\alpha$ which is eligible for an equilibrium solution of the mean field game. 

\begin{proposition} \label{prop:find_pi}
    Let $\mu \in \Pc_2^{\F^0}(\Cc)$ and $\alpha$ be a scalar $(\Fc_t^0 \otimes \Bc(\R))_{t \in [0,T]}$--progressively measurable process satisfying 
\be \label{eq:condition_sigma0}
\E\left[ \widehat{\E}^{\mu} \int_0^T \alpha_t(\widehat{X}_t)^2 (1+|\lambda_t|^2) \mathrm{d}t \right]< \infty
&\mbox{and}&
\widehat{\E}^{\mu}\left[\sigma^0_t(\widehat{X}_t) \right]= \widehat{\E}^{\mu}\left[\alpha_t(\widehat{X}_t) \right] \neq 0.
\ee 
For an arbitrary scalar $\F^0$--prog. meas. process $\kappa\neq-1$, $\P-$a.s. let
    \b*
        \pi_t(x,y):=\frac{(1+\kappa_t) \alpha_t(x) - \sigma^0_t(x)}{\widehat{\E}^{\mu}[\alpha_t(\widehat{X}_t)]}
        &x,\hat x\in\R,&
        t\in[0,T].
    \e* 
Then, $\pi \in \Ac(\mu,\pi)$ and
    \begin{align*}
        \alpha_t(x)=\Sigma^0_t(x)=A^{\mu,\pi}_t(x,\pi_t)\;\;\mbox{and}\;\;m^{\mu,\pi}_t(x)=(1+\kappa_t)^{-1},
        ~\mbox{for all}~
        (t,x) \in [0,T] \x \R.
    \end{align*}
\end{proposition}

\begin{proof}
    We just need to check the condition with $\pi$ given in the statement. Using the condition \eqref{eq:condition_sigma0}, with $\Sigma^0_t(x):=\alpha_t(x)$, it is straightforward that $\alpha_t(x)=A^{\mu,\pi}_t(x,\pi_t)\;\;\mbox{and}\;\;m^{\mu,\pi}_t(x)=(1+\kappa_t)^{-1}$. By the integrability condition on $\alpha$ and $\mu$, we check that $\pi \in \Ac(\mu,\pi)$.
\end{proof}

\ep

\section{Black--Scholes idiosyncratic dynamics examples}

In this section, we examine the case where the idiosyncratic risk process $P$ of \eqref{SDE:P} is defined by:
\be\label{BS}
dP_t
&=&
X_t\big[ \sigma^0_t (\lambda_t\drm t + \drm W^0_t)
             +\sigma_t \drm W_t
      \big],
~t\in[0,T],
\ee
where we have already accounted for the restriction on the drift $b=\lambda\sigma^0$, for some $\F^0$--progressively measurable process $\lambda$, due the NIP condition. Here, $\sigma^0$ and $\sigma$ are bounded positive $\F^0$--progressively measurable processes, and we assume in addition that 
\be\label{lambda}
Z_T:=\Ec\Big(\int_0^.\lambda_t\;\drm W^0_t\Big)_T
&\mbox{satisfies}&
\E[Z_T]=1
~\mbox{and}~
\E[Z_T^2]<\infty,
\ee 
with $\Ec(M)_t:=e^{M_t-\frac12\langle M\rangle_t}$ denoting the Dol\'eans--Dade exponential for all semimartingale $M$.
 
Our main results in this section provides a class of equilibrium solutions to the MFG of cross--holding under logarithmic and power utility functions.

Observe that we have deliberately omitted here the possible dependence of the coefficients $\lambda,\sigma,\sigma^0$ on the marginal distribution $\mu_t$. We emphasize that there would be no technical additional difficulties added by including the mean field interaction through such a dependence in the coefficients. However, our choice is motivated by the following two important reasons:
\begin{itemize}
\item first, the presentation is much simpler as we do not have to carry everywhere the dependence of the coefficients on the marginal law;
\item second, and more importantly, the problem reduction of Proposition \ref{prop:equivence_MFG} shows that the cross--holding feature of our problem does not induce an additional mean field interaction at the level of the equilibrium dynamics. This is in contrast with the no--common noise setting of Djete \& Touzi \cite{DjeteTouzi} where the cross--holding equilibrium strategy is the source of a mean field interaction in the equilibrium dynamics. Notice the major difference between our context and \cite{DjeteTouzi} that no portfolio constraints are considered in the present problem as we are focusing on the mean--variance tradeoff only. 
\end{itemize}

\subsection{The logarithmic utility case}

Let us consider the special case
\be\label{log}
U(x):=\log{(x)},
&x>0,&
\mbox{with convention}~U=-\infty~\mbox{on}~(-\infty,0].
\ee

\begin{proposition}\label{prop:log}
Under \eqref{BS}, \eqref{lambda} and \eqref{log}, let $\kappa$ be an arbitrary $\F^0-$progressively measurable process with $1+\kappa_t>0$, and let $X$ be an It\^o process in $\Sc^2$ defined by
\be\label{mu:log}
\frac{\drm X_t}{X_t}
&=&
\sigma^0_t (\lambda_t\drm t + \drm W^0_t)
        +\frac{\sigma_t}{1+\kappa_t} \drm W_t,
~~\P-\mbox{a.s.}
\ee
denote $\mu^\kappa:=\Lc_{X|W^0}$ the $($log--gaussian$)$ conditional law of $X$ given $\F^{W^0}$. Then 

\medskip
{\rm (i)} there exists a solution of the mean field game of cross--holding with equilibrium distribution $\mu^\kappa$ if and only if $\lambda=\sigma^0$; 

\medskip
{\rm (ii)} and in this case any strategy $\pi^\kappa\in\Ac(\mu,\pi^\kappa)$ satisfying 
\be\label{pi:log}
\widehat\E^{\mu^\kappa}\left[\pi^\kappa_t(X_t,\widehat X_t)\widehat X_t \right]=\kappa_tX_t,
&\mbox{and}&
\widehat\E^{\mu^\kappa}\left[\pi^\kappa_t(X_t, \widehat X_t) \right]=\kappa_t,
~~\P-\mbox{a.s.}
\ee
is a corresponding equilibrium cross--holding strategy.
\end{proposition}

\proof We use the equivalence in \Cref{prop:equivence_MFG} and argue in three steps.

\medskip
{\bf 1.} Given that the log--utility function is negative infinite on $(-\infty,0]$, we may restrict the representative agent utility maximization step in the MFG definition to those strategies which induce strictly positive equity value at the final time $T$. Now, as $\E[Z_T]=1$ by \eqref{lambda}, the equity process $X$ is a local martingale under the equivalent probability measure $\Q:=Z_T\cdot\P$. Moreover $\E^\Q[\sup_{t\le T}|X_t|]\le\|Z\|_{\L^2} \big\|\sup_{t\le T}|X_t|\big\|_{\L^2}<\infty$ by the second condition in \eqref{lambda} and the admissibility condition of any MFG solutions. It follows that the equity process $X$ is a $\Q-$martingale, and consequently 
\be\label{X>0}
    X_t=\E^\Q[X_T|\Fc_t] \ge 0
    &\mbox{for all}&
    t\in[0,T],~\P\mbox{--a.s.}
\ee
If $\P \left[  \E^\Q[X_T|\Fc_t] =0 \right] >0$, we have $X_T=0$ on $\{ \E^\Q[X_T|\Fc_t] =0\}$. Since $\sup_{t\le T}|X_t| < \infty$ a.e., $U(X_T) < \infty$ a.e., this leads to  
$$
    U(X_T)= U(X_T)\1_{\{\E^\Q[X_T|\Fc_t] =0\}} + U(X_T)\1_{\Om \setminus \{\E^\Q[X_T|\Fc_t] =0\} }=-\infty.
$$ 
Therefore, any MFG solutions must verify: for all $t \in [0,T]$, $X_t >0$, $\P$--a.e. and we may restrict the expected utility maximization to the strategies verifying this condition.

\medskip
\noindent {\bf 2.} Next, identifying the $\drm W_t-$coefficient of SDE \eqref{mu:log} with that in the dynamics of the equilibrium equity process \eqref{X:alpha}, we see that whenever $(\mu,\pi)$ is an equilibrium solution of the MFG, we have 
$$
m^{\mu,\pi}_t(X^{\mu,\pi,\pi}_t)
=
\frac{1}{1+\kappa_t},
~\P-\mbox{a.s. and then}~
m^{\mu,\pi}_t(x)
=
\frac{1}{1+\kappa_t},~\P-\mbox{a.s. for a.e.}~x>0.
$$
This already shows that the second equation in \eqref{pi:log} holds true. On the other hand, together with \eqref{X>0}, this allows allows to rewrite the SDE of $\X^\alpha$ as 
\begin{equation}\label{SDEgeom}
\frac{\drm\X^\alpha_t}{\X^\alpha_t}
=
a_t(\X^\alpha_t)(\lambda_t\drm t+\drm W^0_t)
+v_t\drm W_t,
~\mbox{with}~
\alpha_t(\X^\alpha_t)
=\X^\alpha_ta_t(\X^\alpha_t),
~v_t:=\frac{\sigma_t}{1+\kappa_t}.
\end{equation}
{\bf 3.} Next, given our admissibility conditions on $\alpha$, we compute for the present exponential utility context \eqref{log} that
$$
\J_{\mu,\pi}(\alpha)
=
\log{(X_0)}
+\E\Big[\int_0^T \Big(a_t(\X^\alpha_t)\lambda_t
                                  -\frac12a_t(\X^\alpha_t)^2
                                  -\frac12v_t^2
                          \Big)\drm t
      \Big],
$$
which obviously achieves a maximum at $\hat a_t(\X^\alpha_t)=\lambda_t$. Then,
\b*
\frac{\drm\X^{\hat\alpha}_t}{\X^{\hat\alpha}_t}
=
\lambda_t(\lambda_t\drm t+\drm W^0_t)
+v_t\drm W_t,
\e*
which compared with \eqref{mu:log} implies that $\lambda=\sigma^0$ is necessary for the dynamics of $\X$ to qualify for an equilibrium solution of the mean field game. Finally, identifying again the last equilibrium dynamics with \eqref{X:alpha}, we see that:
$$
\sigma^0_tX_t
=
\frac{X_t+\widehat\E^{\mu^\kappa}\left[\pi_t(X_t,\widehat X_t)\widehat X_t\right]}
       {1+\kappa_t},
$$
which provides the first equation of \eqref{pi:log}.
\ep

\subsection{The power utility case}

Similar to the previous section, we consider the case of power utility for the representative agent
\be\label{power}
U(x):=\frac{x^p}{p},
&x>0,&
\mbox{with convention}~U=-\infty~\mbox{on}~(-\infty,0].
\ee

\begin{proposition}
Under \eqref{BS}, \eqref{lambda} and \eqref{power}, let $\kappa$ be an arbitrary $\F^0-$progressively measurable process with $1+\kappa_t>0$, and 
\be\label{QBSDE}
\E\Big[e^{\frac{2-p}{1-p}|\xi^\kappa|}
   \Big]<\infty
&\mbox{with}&
\xi^\kappa:=p\int_0^T\Big(\frac{\lambda_r^2}{1-p}
                                -\frac{(1-p)\sigma_r^2}{(1+\kappa_r)^2}
                                \Big) \mathrm{d}r.
\ee
Let $X$ be an It\^o process in $\Sc^2$ defined by
\be\label{mu:power}
\frac{\drm X_t}{X_t}
&=&
\sigma^0_t (\lambda_t\drm t + \drm W^0_t)
        +\frac{\sigma_t}{1+\kappa_t} \drm W_t,
~~\P-\mbox{a.s.}
\ee
and denote $\mu^\kappa:=\Lc_{X|W^0}$ the (gaussian) conditional law of $X$ given $\F^{W^0}$. Then

\medskip
{\rm (i)} there exists a solution of the mean field game of cross-holding with equilibrium distribution $\mu^\kappa$ if and only if 
$
(1-p)\sigma^0
=
\lambda+Z^\kappa,
$ where $(Y^\kappa,Z^\kappa)$ is the unique solution of the quadratic backward SDE 
\be\label{QBSDE}
Y^\kappa_t
=
\frac12\int_t^T \Big((Z^\kappa_r)^2
                                +\frac{p(\lambda_r+Z_r^\kappa)^2}{1-p}
                                -\frac{p(1-p)\sigma_r^2}{(1+\kappa_r)^2}
                                \Big)\mathrm{d}r
 -\int_t^TZ^\kappa_r \mathrm{d}W^0_r,
 ~t\le T;
\ee

\medskip
{\rm (ii)} In this case any strategy $\pi^\kappa\in\Ac(\mu,\pi^\kappa)$ satisfying 
\be\label{pi:power}
\widehat\E^{\mu^\kappa}\left[\pi^\kappa_t(X_t,\widehat X_t)\widehat X_t \right]=\kappa_tX_t,
&\mbox{and}&
\widehat\E^{\mu^\kappa}\left[\pi^\kappa_t(X_t, \widehat X_t) \right]=\kappa_t,
~~\P-\mbox{a.s.}
\ee
is a corresponding equilibrium cross--holding strategy.
\end{proposition}

\proof We first follow exactly the same line of argument as in the proof of Proposition \ref{prop:log} to rewrite the reduced state process $\X$ in the geometric form \eqref{SDEgeom}. 

We next focus on the representative agent optimization step of the MFG problem. By Briand \& Hu \cite{briand2008quadratic}, the quadratic BSDE \eqref{QBSDE} has a solution $(Y^\kappa,Z^\kappa)$ with
\be\label{Zkappa}
\E\Big[\int_0^T (Z^\kappa_t)^2\mathrm{d}t\Big]
&<&
\infty.
\ee
The uniqueness of such a solution is a consequence of the following argument.

Denote by $H$ the generator of the quadratic backward SDE \eqref{QBSDE}, and set $v:=\frac{\sigma}{1+\kappa}$. Then, for all $\alpha$ satisfying the admissibility condition \eqref{alpha},  it follows from It\^o's formula that the dynamics of process $V^\alpha_t:=U(\X^\alpha_t)e^{Y^\kappa_t}$ are given by:  
\b*
\mathrm{d}V^\alpha_t
&\!\!\!=&\!\!\!
e^{Y^\kappa_t}\big[(\X^\alpha_t)^{p-1}\mathrm{d}\X^\alpha_t
                  +\frac12(p-1)(\X^\alpha_t)^{p-2}\mathrm{d}\langle \X^\alpha\rangle_t
            \big]
\\
&\!\!\!&\!\!\!
+\frac1p(\X^\alpha_t)^pe^{Y^\kappa_t}
   \big[\mathrm{d}Y^\kappa_t+\frac12\mathrm{d}\langle Y^\alpha\rangle_t
   \big]
+(\X^\alpha_t)^{p-1}e^{Y^\kappa_t}\mathrm{d}\langle \X^\alpha,Y^\kappa\rangle_t
\\
&\!\!\!=&\!\!\!
e^{Y^\kappa_t}\big[(\X^\alpha_t)^{p}(a_t(\lambda_t\mathrm{d}t+\mathrm{d}W^0_t)
                                                           +v_t\mathrm{d}W_t
                                                          )
                  +\frac12(p-1)(\X^\alpha_t)^{p}(a_t^2+v_t^2)\mathrm{d}t
            \big]
\\
&\!\!\!&\!\!\!
+\frac1p(\X^\alpha_t)^pe^{Y^\kappa_t}
   \big[-H_t\mathrm{d}t+Z^\kappa_t\mathrm{d}W^0_t
          +\frac12(Z^\kappa_t)^2\mathrm{d}t
   \big]
+(\X^\alpha_t)^{p}e^{Y^\kappa_t} a_tZ^\kappa_t \mathrm{d}t
\\
&\!\!\!=&\!\!\!
V_t^\alpha
\Big[ \Big(\!-\!H_t
                +pa_t\lambda_t+\frac12p(p-1)(a_t^2+v_t^2)
                +\frac12(Z^\kappa_t)^2
                +pa_tZ_t^\kappa
        \Big)\mathrm{d}t
+(pa_t+Z^\kappa_t)\mathrm{d}W^0_t
\Big].
\e*
By the definition of the generator $H$, we see that 
\begin{itemize}
\item for any control process $\alpha$, the drift in the previous dynamics is nonpositive, so that $V_t^\alpha$ is a nonnegative local supermartingale and therefore a supermartingale, implying that 
$$
U(\X_0)e^{Y_0^\kappa}
= 
V_0
\ge
\E \big[V^\alpha_T\big]
=
\E \big[U(\X^\alpha_T)\big],
~\mbox{for all}~
\alpha;
$$
\item while for the optimal choice of $\hat a:=\frac{\lambda_t+Z_t^\kappa}{1-p}$, the drift in the previous dynamics is zero, so that $V_t^\alpha$ is a local martingale, and a martingale due to the integrability condition \eqref{Zkappa} on $Z^\kappa$ and the expression of $\hat a$, implying that
$$
V_0
=
\E \big[V^{\hat\alpha}_T\big]
=
\E \big[U(\X^{\hat\alpha}_T)\big],
~\mbox{with}~
\hat\alpha_t(x):=\hat a_tx.
$$
\end{itemize}
This shows on one hand that $V_0=\sup_\alpha J_{\mu,\pi}(\alpha)$, and in fact the same argument started at the time origin $t$ induces also a unique representation of $V_t$ in terms of the similar control problem started from time $t$. Hence the uniqueness of $Y^\kappa$, and the uniqueness of $Z^\kappa_t=\frac1{\drm t}\langle Y^\kappa,W^0\rangle_t$, $\P-$a.s. follows immediately. 

The last argument also provides as a by-product the optimality of $\hat\alpha$, and we may now write the dynamics of the state under this optimal policy:
\b*
\frac{\drm\X^{\hat\alpha}_t}{\X^{\hat\alpha}_t}
&=&
\frac{\lambda_t+Z_t^\kappa}{1-p}(\lambda_t\drm t+\drm W^0_t)
+ \frac{\sigma_t}{1+\kappa_t}\drm W_t,
\e*
which is consistent with \eqref{mu:power} if and only if $
(1-p)\sigma^0
=
\lambda+Z^\kappa
$. This completes the proof of (i), and (ii) follows exactly the same line of argument as in the proof of Proposition \ref{prop:log}.
\ep

\quad

\begin{remark}
{\rm (i)} In light of the two previous examples, by using {\rm \Cref{prop:find_pi}}, we can see that a cross--holding strategy is given by 
\b*
\pi_t(x,\widehat{x})=\frac{\kappa_t x}{\E^{\mu^\kappa}[\widehat{X}_t ]},
&\mbox{for all}&
x,\hat x\in\R.
\e*
{\rm (ii)} However, we may search for solutions of the form $\pi_t(x,\widehat{x})=\sum_{i=1}^Nf^{(i)}_t(x)\hat f^{(i)}_t(\widehat{x})$ for some $N\in\N\cup\{\infty\}$. Let us develop here the calculations for the case $N=2$:
\b*
\pi_t(x,\hat x)
=
f_t(x)\hat f_t(\hat x)+g_t(x)\hat g_t(\hat x),
&\mbox{for all}&
x,\hat x\in\R.
\e*
with the normalization
\b*
\E^{\mu^\kappa}[f_t(\widehat{X}_t)]
&=&
\E^{\mu^\kappa}[g_t(\widehat{X}_t)]=1.
\e*
In this case, the second equation of \eqref{pi:log} says that $\hat f_t(\hat x)+\hat g_t(\hat x)
=
\kappa_t$, which induces by injecting in the first equation that
\b*
(f_t-g_t)(x)\E^{\mu^\kappa}[\hat f_t(\widehat{X}_t)\widehat{X}_t]
&=&
\kappa_t\big(x-g(x)\E^{\mu^\kappa}[\widehat{X}_t]\big).
\e*
Consequently, we obtain a class of equilibrium cross-holding strategies parameterized by arbitrary random maps $\hat f$ and $g$ with $\E^{\mu^\kappa}[g(\widehat{X}_t)]=1$:
\be
\pi_t(x,\hat x)
&=&
f_t(x)\hat f_t(\hat x)+g_t(x)[\kappa_t-\hat f_t(\hat x)]
\nonumber\\
&=&
\kappa_tg_t(x)+(f_t-g_t)(x)\hat f_t(\hat x)
\nonumber\\
&=&
\kappa_t\Big[g_t(x)
                     + \big(x-g_t(x)\E^{\mu^\kappa}[\widehat{X}_t]\big)
                        \frac{\hat f_t(\hat x)}{\E^{\mu^\kappa}[\hat f_t(\widehat{X}_t)\widehat{X}_t]}
              \Big].
\label{pi:class2}
\ee
\end{remark}

\begin{remark}$($Interpretation of the optimal strategy$)$ 

$(i)$ The item $(i)$ of the previous remark shows that an optimal strategy consists to buy or sell a number which is proportional to the ratio of the current position of the player by the mean/average of the positions of the population/other players. Thus, each player has an idea of his proximity to the optimal by comparing his position to the average position of the others.

\medskip
$(ii)$ We can observe that, in the case where $\kappa$ is positive, following the best strategy makes it possible to reduce volatility. Indeed, in the case where the provision is given by \eqref{BS}, doing nothing generates $X_t\left[ \sigma^0_t\drm W^0_t+\sigma_t \drm W_t \right]$ in the diffusion part, while that by following the optimal strategy, we get $X_t\left[ \sigma^0_t\drm W^0_t+\frac{\sigma_t}{1+\kappa_t} \drm W_t \right]$ for the diffusion part. The diffusion coefficient of the idiosyncratic noise is divided by $(1+ \kappa_t)$. 

\medskip
$(iii)$ Notice that, by choosing $\kappa =0$, we obtain an optimal strategy by taking the identical nul control i.e. $\pi_t(x,\widehat x)=0$. In other words, we can reach the optimal by doing nothing i.e. no buy no sell. From the initial formulation of our MFG formulation or from the $n$--player game formulation, this conclusion seems hard to observe at first sight.
    
\end{remark}

\begin{appendix}

\section{Appendix: proofs in the one--period MFG} \label{preuves_section2}

\noindent {\bf Proof of Lemma \ref{lemme_fermeture}}
First, notice that  $F^{\mu}=0,$ a.s implies $G^\beta_{\mu}=0$, a.s for all $\beta \in \Ac_1(\F^\mu,\pi)$, then $K=\L^-_1(\mu_0 \otimes \rho)$ is a closed subset.  Next,  let $F^\mu \in \L_1(\mu_0 \otimes \rho)$ such that  
\begin{equation}
     (\mu_0 \otimes \rho)[F^\mu \ne 0] >0.
    \label{condition F_mu}
\end{equation}

\noindent Let $(\xi_n=\xi_n(X_0,\varepsilon^0))_{n \in \N}$ be some sequence in $K$ converging to some $\xi=\xi(X_0,\varepsilon^0)$, i.e 
\be\label{limite}\label{inegalite_suite}
\xi_n \leq G_{ \mu}^{\beta_n},~\mbox{a.s for some}~
\beta_n \in \Ac_1(F^{\mu},\pi),
~\mbox{and}~
\xi_n 
&\longrightarrow& 
\xi
~\mbox{in}~\L^1,
\ee
and let us show that $\xi \leq G_{ \mu}^\beta$, $\mu_0 \otimes \rho-\mbox{a.s}$, for some $ \beta \in \Ac_1(F^{\mu},\pi)$. We argue separately the two following alternative cases.
\begin{enumerate}
\item If $\liminf_{n} \|\beta_{n}\|_{\L^\infty}<\infty$, then by Mazur's Lemma, see e.g. Theorem III.7 in Brezis \cite{brezis}, we may find $\widehat\beta_n\in\mbox{Conv}(\beta_k,k\ge n)$ such that $\widehat\beta_n\longrightarrow\beta$ in $\L^\infty$. In particular,
$$
\int_{\R} \!\!\big|(\widehat\beta_n-\beta)(x,\widehat x)
              F^\mu(\widehat x,y)\big|\mu(\mathrm{d}x)\mu(\mathrm{d}\widehat x)\rho(\mathrm{d}y)
\le
\|\widehat\beta_n-\beta\|_{\L^\infty}\|F^\mu\|_{\L^1}
\to
0,~\mbox{as}~n\to\infty,
$$
so that $G_\mu^{\hat\beta^n}\longrightarrow G_\mu^\beta$ in $\L^1$. Then, taking the appropriate convex combinations on both sides of the first inequality of \eqref{limite} and passing to the limit $n\to\infty$ then induces $\xi\le G_\mu^\beta$, as required.

\item In the alternative case $\lim_{n} \|\beta_{n}\|_{\L^\infty}=\infty$, we introduce the sequence $\gamma_n:=\frac12\frac{\beta_n}
                                                       {\|\beta_n\|_{\L^\infty}} 
                                      \1_{\{\| \beta_n\|_{ \L^\infty}\ne 0 \}}$, $n \ge 1,$ which satisfies $\sup_n \|\gamma_n\|_{\L^\infty} \leq \frac{1}{2}$. Similar to the previous step, we may find a sequence $\hat\gamma_n\in\mbox{Conv}(\gamma_k,k\ge n)$ such that $\hat\gamma_n\longrightarrow\gamma$ in $\L^\infty$, with $\|\gamma\|_{\L^\infty}\le \frac12$, and $G_\mu^{\hat\gamma_n}\longrightarrow G_\mu^{\gamma}$ in $\L^1$. Dividing both sides of \eqref{limite} by $2\|\beta_n\|_{\L^\infty}$ and taking the appropriate convex combinations on both sides of the first inequality, we obtain by sending $n\to\infty$ that $0\le G_\mu^\gamma$, a.s. due to the fact that $\|\beta_n\|_{\L^\infty}\longrightarrow\infty$. By the no--arbitrage condition, this implies that $0=G_\mu^\gamma=\int_{\R} \gamma(x,\widehat x)F^\mu(\widehat x,\eps^0)\mu_0(\mathrm{d}\widehat x)$, a.s. Now, since $\|\gamma\|_{\L^\infty}\le \frac12$, this Fredholm integral equation has a unique solution $F^\mu(X_0,\eps^0)=0,$  a.s (see T. Kato \cite{kato_1980} page 153 by considering the operator $T_\gamma: f \mapsto \int_{\R} \gamma(x,\hat x)f(\hat x,\varepsilon^0)\mu_0(\mathrm{d} \hat x)$). This contradicts  condition \eqref{condition F_mu}.
\end{enumerate}
\ep

\medskip
\noindent {\bf Proof of Theorem \ref{Theorem_eq_Nash1} }\label{preuve_theorem}
The proof is organized in two steps. 

\medskip
\noindent {\it Step 1: } We first show that, under the uniqueness of the Fredholm equation \eqref{eq_Nash1}, the random environment~$(F^{\mu}, \pi) \in \L^2(\mu_0 \otimes \rho) \times  \in \L^2(\mu_0 \otimes \mu_0)$  is solution to the MFG problem if and only if the following system is satisfied
\begin{equation}
\begin{aligned}
 F^{\mu}(X_0,\varepsilon^0)&= F^\mu_0(X_0)+F_1^\mu(X_0) \varepsilon^0, \, \, \mu_0 \otimes \rho\mbox{--a.s},  \\
  q F^\mu_0(\widehat X_0)&= F_1^\mu(X_0) F_1^\mu(\widehat X_0) \, \, \mu_0 \otimes \mu_0\mbox{--a.s},
\end{aligned}
\label{eq_Nash}
\end{equation} where $F^\mu_0(X_0)=\E^\rho\big[ F^{\mu}(X_0, \varepsilon^0)\big]$ and $F_1^\mu \in \L^2(\mu_0)$ satisfy respectively  
\begin{eqnarray}
    \frac{F^\mu_0(X_0)}{m^{\pi}(X_0)}&=&b(X_0)+\widehat \E^{\mu_0}\big[ \pi(X_0,\widehat X_0) F^\mu_0(\widehat X_0)\big], \mu_0\mbox{--a.s}.\label{F0}\\
    \frac{F_1^\mu(X_0)}{m^{\pi}(X_0)}&=&\sigma(X_0)+\widehat \E^{\mu_0}\big[ \pi(X_0,\widehat X_0) F_1^\mu(\widehat X_0)\big], \mu_0\mbox{--a.s}. 
    \label{Fredholm_G}
\end{eqnarray}

\medskip
While the first equation in \eqref{eq_Nash} was already obtained in \eqref{FG}, the second one is equivalent to the individual optimization step (i) of Definition \ref{def:MFG1}. To prove it, we first notice that $\frac{1}{m^{\pi}(X_0)} \ne 0$, $\mu_0\mbox{--a.s}$. 
Indeed, suppose on the contrary that $\frac{1}{m^{\pi}}=0$ on some subset $A \subset \supp(\mu_0)$ with $ \mu_0(A)>0$, then the mean field equation \eqref{1periodMF1} on $A$ is reduced to
\begin{equation*} 
\int_\R \pi(x_0, y) F^\mu(y,\varepsilon^0)\mu_0(\mathrm{d}y) +b(x_0)+\sigma^0(x_0)\varepsilon^0+\sigma(x_0) \varepsilon=0, \,\rho\mbox{--a.s}\, \mbox{ for all } \, x_0 \in A. \label{m_0}
\end{equation*}
Since $\varepsilon$ is centered and independent of $(X_0,\eps^0)$, this is in contradiction with our non--degeneracy assumption $\sigma(X_0)\ne 0$, $ \mu_0\mbox{--a.s}$. 

\smallskip
We next turn to the characterization of the representative agent optimization problem of the mean--variance criterion:
\b*
J(\beta)
&=&
X_0 + m^\pi(X_0) \Big( b(X_0)
                                     + \int_\R\int_\R \beta(X_0, y)\big(F^\mu_0(y)+zF_1^\mu(y) \big) \mu_0(\mathrm{d}y)
                                              \rho( \mathrm{d} z) 
                            \Big)
\\
&&
-\frac{\sigma(X_0)^2}{2q}
-\frac{m^\pi(X_0)^2}{2q}\int_\R \!\!\Big( \!\int_\R \!\!\beta(X_0,y)zF_1^\mu(y)\mu_0(\mathrm{d} y)
                                                      +\sigma^0(X_0)z 
                                              \Big)^2 \rho(\mathrm{d} z). 
\e*
Since~$q>0$, it is clear that $J$ is a strictly concave function, and that the optimality is characterized by the first order condition which follows from direct calculation of the G\^ateaux derivative of the functional $J$: 
\b*
 qF^\mu_0(\widehat X_0)- F_1^\mu(\widehat X_0) 
 m^\pi(X_0)\big(\sigma^0(X_0)+\int_\R  \pi(X_0,u)F_1^\mu(u)\mu_0( \mathrm{d} u)\big)=0, \, \, \mbox{a.s.}
\e*
In view of \eqref{Fredholm_G}, the last condition is equivalent to the second one in \eqref{eq_Nash}.

\bigskip \noindent {\it Step 2:} By the symmetry of the second equation of \eqref{eq_Nash},  we see that  $F^\mu_0(X_0)=c$ is a constant, $\mu_0\mbox{--a.s}$, which can be determined by direct substitution in \eqref{F0}. This provides $c\big(\frac{1}{m^\pi(X_0)}-\int \pi(X_0,y) \mu_0(dy) \big)=b(X_0)$, a.s. and therefore $c=\E[b(X_0)]$ by taking expectations on both sides. This shows that \eqref{eq_Nash} is equivalent to  
\begin{equation}
\begin{aligned}
 F^{\mu}(X_0,\varepsilon^0)&= \E[b(X_0)]+F_1^\mu(X_0) \varepsilon^0 \, \, \mu_0 \otimes \rho-\mbox{a.s},  \\
  q\E[b(X_0)]&=F_1^\mu(X_0)F_1^\mu(\widehat X_0)\, \, \mu_0\otimes \mu_0-\mbox{a.s},
\end{aligned}
\label{eq_Nash2}
\end{equation} with the restriction \eqref{drift_b} and with $F_1^\mu$ satisfying \eqref{Fredholm_G}. 
As $X_0$ and $\widehat X_0$ are identically distributed,  and by taking the expectation on $F_1^\mu(X_0)$ and $F_1^\mu(X_0)^2$, we see that $\E^{\mu_0}\big[F_1^\mu(X_0)\big]^2=\E^{\mu_0}\big[F_1^\mu(X_0)^2\big]=q \E[b(X_0)].$ Then $\V\mbox{ar}^{\mu_0}[F_1^\mu(X_0)]=0$, and we get 
\b*
\mbox{either}~F_1^\mu(X_0)= \sqrt{q \E[b(X_0)]}, 
~\mu_0\mbox{--a.s.} 
&\mbox{or}&
F_1^\mu(X_0)=-\sqrt{q \E[b(X_0)]},~\mu_0\mbox{--a.s.}
\e*
If $F_1^\mu(X_0)=\sqrt{q \E[b(X_0)]}$, a.s. then the Fredholm  equation \eqref{Fredholm_G} is equivalent to 
$$ \sigma^0(X_0)= \sqrt{q \E[b(X_0)]} \Big( \frac{1}{m^\pi(X_0)}-\widehat \E^{\mu_0}\big[ \pi(X_0,\widehat X_0)\big] \Big) , \, \, \mu_0\mbox{--a.s},  $$
Since $\sigma^0(X_0) \ne 0$, $\mu_0\mbox{--a.s}$ and the drift $b$ satisfies \eqref{drift_b}, the  last equation is equivalent to~$\E[b(X_0)] >0$ and $\sigma^0=\sqrt{\frac{q}{\E[ b(X_0) ]}} b$.  

The alternative case $F_1^\mu(X_0)=-\sqrt{q \E[b(X_0)]}$ is studied similarly.

\ep

\end{appendix}

\bibliographystyle{plain}

\bibliography{MFG_Nizar-Bassou_ArxivVersion}


 






\end{document}